\documentclass[reqno]{amsart}
\usepackage[latin9]{inputenc}
\usepackage{xcolor}
\usepackage{enumitem}
\usepackage{amstext}
\usepackage{amsthm}
\usepackage{amssymb}
\usepackage{esint}
\PassOptionsToPackage{normalem}{ulem}
\usepackage{ulem}
\usepackage[unicode=true,pdfusetitle,
 bookmarks=true,bookmarksnumbered=false,bookmarksopen=false,
 breaklinks=false,pdfborder={0 0 0},pdfborderstyle={},backref=false,colorlinks=true]
 {hyperref}

\makeatletter

\providecommand{\tabularnewline}{\\}
\providecolor{lyxadded}{rgb}{0,0,1}
\providecolor{lyxdeleted}{rgb}{1,0,0}

\DeclareRobustCommand{\lyxsout}[1]{\ifx\\#1\else\sout{#1}\fi}

\numberwithin{equation}{section}
\numberwithin{figure}{section}
\theoremstyle{plain}
\newtheorem{thm}{\protect\theoremname}
\theoremstyle{definition}
\newtheorem{defn}[thm]{\protect\definitionname}
\theoremstyle{remark}
\newtheorem{rem}[thm]{\protect\remarkname}
\theoremstyle{plain}
\newtheorem{lem}[thm]{\protect\lemmaname}
\theoremstyle{definition}
\newtheorem{example}[thm]{\protect\examplename}
\theoremstyle{plain}
\newtheorem{cor}[thm]{\protect\corollaryname}

\usepackage{pdfsync}
\usepackage{graphics}
\usepackage{epstopdf}
\usepackage[all]{xy}
\usepackage{amscd}
\usepackage{enumitem}
\usepackage{mathtools}
\setlist[enumerate]{leftmargin=*,label=(\roman*),align=left}

\makeatletter
\@namedef{subjclassname@2020}{%
  \textup{2020} Mathematics Subject Classification}
\makeatother

\newcommand{\xyR}[1]{ \makeatletter
\xydef@\xymatrixrowsep@{#1} \makeatother} 
\newcommand{\xyC}[1]{ \makeatletter
\xydef@\xymatrixcolsep@{#1} \makeatother} 
\entrymodifiers={++[ ][F]} 

\newcommand{\ra}{\longrightarrow}

\newcommand{\field}[1]{\mathbb{#1}}
\newcommand{\R}{\field{R}} 
\newcommand{\N}{\field{N}} 

\newcommand{\eps}{\varepsilon} 
\renewcommand{\phi}{\varphi}
\newcommand{\diff}[1]{\ifmmode\mathchoice{\hbox{\rm d}#1}  
 {\hbox{\rm d}#1}  
 {\scalebox{0.75}{$\hbox{\rm d}#1$}}  
 {\scalebox{0.35}{$\hbox{\rm d}#1$}}  
 \fi} 

\newcommand{\abs}[2][\empty]{\ifx#1\empty\left|#2\right|%
\else#1\vert #2 #1\vert\fi}

\newcommand{\Coo}{\mbox{\ensuremath{\mathcal{C}}}^{\infty}} 

\newcommand{\Rtil}{\widetilde \R} 

\newcommand{\frontRise}[2]{\ifmmode\mathchoice{{\vphantom{#1}}^{\scalebox{0.6}{$#2$}}}  
 {{\vphantom{#1}}^{\scalebox{0.56}{$#2$}}}  
 {{\vphantom{#1}}^{\scalebox{0.47}{$#2$}}}  
 {{\vphantom{#1}}^{\scalebox{0.35}{$#2$}}}\fi} 
\newcommand{\RC}[1]{\frontRise{\R}{#1}\Rtil}
\newcommand{\rcrho}{\RC{\rho}}
\newcommand{\rti}{\RC{\rho}}

\newcommand{\hyperN}[1]{	\frontRise{\N}{#1}\widetilde{\N}}
\newcommand{\hypNr}{\hyperN{\rho}}
\newcommand{\hypNs}{\hyperN{\sigma}}
\newcommand{\nint}{\text{\rm ni}}
\newcommand{\hyperlimarg}[3]{\mathchoice{\frontRise{\lim}{\raisebox{-0.05em}{$#1\hspace{-0.67em}$}}\lim_{#3\in \hyperN{#2}\,}}
{\frontRise{\lim}{#1\hspace{-0.25em}}\lim_{#3\in \hyperN{#2}\,}}
{\frontRise{\lim}{#1\hspace{-0.25em}}\lim_{#3\in \hyperN{#2}\,}}
{\frontRise{\lim}{#1\hspace{-0.25em}}\lim_{#3\in \hyperN{#2}\,}}}
\newcommand{\hyperlim}[2]{\hyperlimarg{#1}{#2}{n}}

\newcommand{\hypersumarg}[3]{\mathchoice{\frontRise{\sum}{\raisebox{-0.2em}{$#1\hspace{-0.67em}$}}\sum_{#3\in \hyperN{#2}\,}}
{\frontRise{\sum}{#1\hspace{-0.25em}}\sum_{#3\in \hyperN{#2}\,}}
{\frontRise{\sum}{#1\hspace{-0.25em}}\sum_{#3\in \hyperN{#2}\,}}
{\frontRise{\sum}{#1\hspace{-0.25em}}\sum_{#3\in \hyperN{#2}\,}}}
\newcommand{\parthypersumarg}[4]{\mathchoice{\frontRise{\sum}{\raisebox{-0.2em}{$#1\hspace{-2.2em}$}}\sum_{#3\in \hyperN{#2}_{#4}\,}}
{\frontRise{\sum}{#1\hspace{-0.25em}}\sum_{#3\in \hyperN{#2}_{#4}\,}}
{\frontRise{\sum}{#1\hspace{-0.25em}}\sum_{#3\in \hyperN{#2}_{#4}\,}}
{\frontRise{\sum}{#1\hspace{-0.25em}}\sum_{#3\in \hyperN{#2}_{#4}\,}}}
\newcommand{\hypersum}[2]{\hypersumarg{#1}{#2}{n}}
\newcommand{\subzero}{\subseteq_{0}}

\newcommand{\rcrhou}{\rcrho_{\text{\rm u}}}

\newcommand{\sbpt}[1]{#1_{\text{\rm s}}}

\newcommand{\frontRiseDown}[3]{\ifmmode\mathchoice{{\vphantom{#1}}^{\scalebox{0.6}{$#2$}}_{\scalebox{0.6}{$#3$}}}  
 {{\vphantom{#1}}^{\scalebox{0.56}{$#2$}}_{\scalebox{0.56}{$#3$}}}  
 {{\vphantom{#1}}^{\scalebox{0.47}{$#2$}}_{\scalebox{0.47}{$#3$}}}  
 {{\vphantom{#1}}^{\scalebox{0.35}{$#2$}}_{\scalebox{0.35}{$#3$}}}\fi} 
\newcommand{\RCud}[2]{\frontRiseDown{\R}{#1}{#2}\Rtil}
\newcommand{\rcrhos}{\RCud{\rho}{\sigma}_{\text{\rm s}}}
\newcommand{\gsfud}[2]{\frontRiseDown{\mathcal{G}}{#1}{#2}\mathcal{GC}^{\infty}}

\newcommand{\DIff}{ \quad\;\; :\!\iff \quad } 

\makeatother

\providecommand{\corollaryname}{Corollary}
\providecommand{\definitionname}{Definition}
\providecommand{\examplename}{Example}
\providecommand{\lemmaname}{Lemma}
\providecommand{\remarkname}{Remark}
\providecommand{\theoremname}{Theorem}

\begin{document}

\title[Hyperseries in the non-Archimedean ring of CGN]{Hyperseries in the non-Archimedean ring of Colombeau generalized
numbers}
\author{Diksha Tiwari \and Paolo Giordano}
\thanks{D.~Tiwari has been supported by grant P 30407 of the Austrian Science
Fund FWF}
\address{\textsc{University of Vienna, Austria}}
\email{diksha.tiwari@univie.ac.at}
\thanks{P.~Giordano has been supported by grants P30407, P34113, P33538 of
the Austrian Science Fund FWF}
\address{\textsc{University of Vienna, Vienna, Austria}}
\email{paolo.giordano@univie.ac.at}
\begin{abstract}
This article is the natural continuation of the paper: Mukhammadiev
A.~et al \emph{Supremum, infimum and hyperlimits of Colombeau generalized
numbers} in this journal. Since the ring $\rcrho$ of Robinson-Colombeau
is non-Archimedean and Cauchy complete, a classical series $\sum_{n=0}^{+\infty}a_{n}$
of generalized numbers $a_{n}\in\rcrho$ is convergent \emph{if} and
only if $a_{n}\to0$ in the sharp topology. Therefore, this property
does not permit us to generalize several classical results, mainly
in the study of analytic generalized functions (as well as, e.g.,
in the study of sigma-additivity in integration of generalized functions).
Introducing the notion of hyperseries, we solve this problem recovering
classical examples of analytic functions as well as several classical
results.
\end{abstract}

\maketitle

\section{Introduction}

In this article, the study of supremum, infimum and hyperlimits of
Colombeau generalized numbers (CGN) we carried out in \cite{MTAG}
is applied to the introduction of a corresponding notion of hyperseries.
In \cite{MTAG}, we recalled that $(x_{n})_{n\in\N}\in\Rtil^{\N}$
is a Cauchy sequence if and only if $\lim_{n\to+\infty}\left|x_{n+1}-x_{n}\right|=0$
(in the sharp topology). As a consequence, a series of CGN
\begin{equation}
\sum_{n\in\N}a_{n}\text{ converges}\ \iff\ a_{n}\to0\text{ (in the sharp topology)}.\label{eq:convNonArch}
\end{equation}
Once again, this is a well-known property of every ultrametric space,
cf., e.g., \cite{Kob96}. The point of view of the present work is
that in a non-Archimedean ring such as $\rcrho$, the notion of hyperseries
$\sum_{n\in\hyperN{\rho}}a_{n}$, i.e.~where we sum over the set
of hyperfinite natural numbers $\hyperN{\rho}$, yields results which
are more closely related to the classical ones, e.g.~in studying
analytic functions, sigma additivity and limit theorems for integral
calculus, or in possible generalization of the Cauchy-Kowalevski theorem
to generalized smooth functions (GSF; see e.g.~\cite{GKV}).

Considering the theory of analytic CGF as developed in \cite{PiScVa09}
for the real case and in \cite{Ver08} for the complex one, it is
worth to mention that several properties have been proved in both
cases: closure with respect to composition, integration over homotopic
paths, Cauchy integral theorem, existence of analytic representatives
$u_{\eps}$, a real analytic CGF is identically zero if it is zero
on a set of positive Lebesgue measure, etc. ~(cf.~\cite{Ver08,PiScVa09,ObPiVa07}
and references therein). On the other hand, even if in \cite{Ver08}
it is also proved that each complex analytic CGF can be written as
a Taylor series, necessarily this result holds only in an infinitesimal
neighborhood of each point. The impossibility to extend this property
to a finite neighborhood is due to \eqref{eq:convNonArch} and is
hence closely related to the approach we follow in the present article.

We refer to \cite{MTAG} for notions such as the ring of Robinson-Colombeau,
subpoints, hypernatural numbers, supremum, infimum and hyperlimits.
See also \cite{LB-Gio17} for a more general approach to asymptotic
gauges. In the present paper, we focus only on examples and properties
related to hyperseries, postponing those about analytic functions,
integral calculus and the Cauchy-Kowalevski theorem to subsequent
works. Once again, the ideas presented in the present article, which
needs only \cite{MTAG} as prior knowledge, can surely be useful to
explore similar ideas in other non-Archimedean Cauchy complete settings,
such as \cite{BeMa19,BeLuB15,Sha13,Kob96}.

\section{\label{sec:Hyperseries}Hyperseries and their basic properties}

\subsection{Definition of hyperfinite sums and hyperseries}

In order to define these notions, the main idea, like in the Archimedean
case $\R$, is to reduce the notion of hyperseries to that of hyperlimit.
Since to take an hyperlimit we have to consider two gauges $\rho$,
$\sigma$, it is hence natural to consider the same assumption for
hyperseries. However, in a non-Archimedean setting, the aforesaid
main idea implies that the main problem is not in defining hyperseries
itself, but already in defining what an hyperfinite sum is. In fact,
if $(a_{i})_{i\in\hypNs}$ is a family of generalized numbers of $\rti$
indexed in $\hypNs$, it is not even clear what $\sum_{i=0}^{4}a_{i}$,
means because there are infinite $i\in\hypNs$ such that $0\le i\le4$.
In general, in a sum of the form $\sum_{i=0}^{N}a_{i}$ with $N\in\hypNs$,
the number of addends depends on how small is the gauge $\sigma$
and hence how large can be $N\le\diff{\sigma}^{-R}$ (for some $R\in\N$).
The problem is simplified if we consider only \emph{ordinary} sequences
$(a_{n})_{n\in\N}$ of CGN in $\rcrho$ and an $\eps$-wise definition
of hyperfinite sum. In order to accomplish this goal, we first need
to extend the sequence
\begin{equation}
\left(N\in\N\mapsto\sum_{n=0}^{N}a_{n}\in\rcrho\right):\N\ra\rcrho\label{eq:extensionProblem}
\end{equation}
of partial sums with summands $a_{n}\in\rcrho$, $n\in\N$, to the
entire set $\hyperN{\sigma}$ of hyperfinite numbers. This problem
is not so easy to solve: in fact, the sequence of representatives
of zero: $a_{n\eps}=0$ if $\eps\le\frac{1}{n}$ and $a_{n\eps}=(1-\eps)^{-n}$
otherwise, where $n\in\N_{>0}$, satisfies
\begin{equation}
\sum_{n=0}^{N_{\eps}}a_{n\eps}=\sum_{n=\left\lceil \frac{1}{\eps}\right\rceil }^{N_{\eps}}\left(\frac{1}{1-\eps}\right)^{n}\quad\forall\eps\in(0,1)_{\R},\label{eq:inftySeries}
\end{equation}
but if $N_{\eps}\to+\infty$ this sum diverges to $+\infty$ because
$\frac{1}{1-\eps}>1$; moreover, for suitable $N_{\eps}$, the net
in \eqref{eq:inftySeries} is of the order of $\left(\frac{1}{1-\eps}\right)^{N_{\eps}}$,
which in general is not $\rho$-moderate. To solve this first problem,
we have at least two possibilities: the first one is to consider a
Robinson-Colombeau ring defined by the index set $\N\times I$ and
ordered by $(n,\eps)\le(m,e)$ if and only if $\eps\le e$. In this
solution, moderate representatives are nets $(a_{n\eps})_{n,\eps}\in\R^{\N\times I}$
satisfying the uniformly moderate condition
\begin{equation}
\exists Q\in\N\,\forall^{0}\eps\,\forall n\in\N:\ |a_{n\eps}|\le\rho_{\eps}^{-Q}.\label{eq:uniformlyModerate}
\end{equation}
Negligible nets are $(a_{n\eps})_{n,\eps}\in\R^{\N\times I}$ such
that
\begin{equation}
\forall q\in\N\,\forall^{0}\eps\,\forall n\in\N:\ |a_{n\eps}|\le\rho_{\eps}^{q}.\label{eq:uniformlyNeglibible}
\end{equation}
Note that, with respect to the aforementioned directed order relation,
for any property $\mathcal{P}$, we have
\[
\left(\exists(n_{0},\eps_{0})\in\N\times I\,\forall(n,\eps)\le(n_{0},\eps_{0}):\ \mathcal{P}\left\{ n,\eps\right\} \right)\ \iff\ \forall^{0}\eps\,\forall n\in\N:\ \mathcal{P}\left\{ n,\eps\right\} .
\]
The main problem with this solution is that it works to define hyperfinite
sums only if
\begin{equation}
\exists Q_{\sigma,\rho}\in\N\,\forall^{0}\eps:\ \sigma_{\eps}\ge\rho_{\eps}^{Q_{\sigma,\rho}}.\label{eq:sigma-rho-min}
\end{equation}
Assume, indeed, that \eqref{eq:uniformlyModerate} holds for all $\eps\le\eps_{0}$,
so that if $N=[N_{\eps}]\in\hypNs$, $N_{\eps}\in\N$, we have
\begin{equation}
\left|\sum_{n=0}^{N_{\eps}}a_{n\eps}\right|\le N_{\eps}\cdot\rho_{\eps}^{-Q}\le\text{\ensuremath{\sigma}}_{\eps}^{-R}\cdot\rho_{\eps}^{-Q}\le\rho_{\eps}^{-RQ_{\sigma,\rho}-Q},\label{eq:hyperfiniteSumR_u}
\end{equation}
where we assumed that $N_{\eps}\le\sigma_{\eps}^{-R}$ for these $\eps\le\eps_{0}$.
This is intuitively natural since the Robinson-Colombeau ring defined
by the aforementioned order relation depends only on the gauge $\rho$,
whereas the moderateness in \eqref{eq:hyperfiniteSumR_u} depends
on the product $N_{\eps}\cdot\rho_{\eps}^{-Q}\le\sigma_{\eps}^{-R}\cdot\rho_{\eps}^{-Q}$
between how many addends we are taking (that depends on $\sigma$)
and the uniform moderateness property \eqref{eq:uniformlyModerate}.

This first solution has three drawbacks: The first one, as explained
above, is that in this ring we can consider hyperfinite sum only if
the relation \eqref{eq:sigma-rho-min} between the two considered
gauges holds. The second one is that we cannot consider divergent
hyperseries such as e.g.~$\hypersum{\rho}{\sigma}n$ because $n\le\rho_{\eps}^{-Q}$
do not hold for $\eps$ small and uniformly for all $n\in\N$. The
third one is that we would like to apply \cite[Thm.~28]{MTAG} to
prove the convergence of hyperseries by starting from the corresponding
converging series of $\eps$-representatives; however, \cite[Thm.~28]{MTAG}
do not allow us to get the limitation \eqref{eq:sigma-rho-min} (and
later, we will see that in general the limitation \eqref{eq:sigma-rho-min}
is impossible to achieve: see just after the proof of Thm.~\ref{thm:epsConv}).

The second possibility to extend \eqref{eq:extensionProblem} to $\hypNs$,
a possibility that depends on two gauges but works for any $\sigma$
and $\rho$, is to say that hyperseries can be computed only for representatives
$(a_{n\eps})_{n\eps}$ which are \emph{moderate over hypersums}, i.e.~to
ask
\begin{equation}
\forall N\in\hyperN{\sigma}:\ \left(\sum_{n=0}^{\nint{(N)}_{\eps}}a_{n\eps}\right)\in\R_{\rho}.\label{eq:moderateHypseries}
\end{equation}
For example, if $k=[k_{\eps}]\in(0,1)\subseteq\rcrho$, then $\left(k_{\eps}^{n}\right)_{n,\eps}$
is moderate over hypersums (see also Example~\ref{exa:GeomExpSeries}
below); but also the aforementioned $a_{n\eps}=n$ for all $n\in\N$
and $\eps\in I=(0,1]_{\R}$ is clearly of the same type if $\R_{\sigma}\subseteq\R_{\rho}$.

However, if $a_{n}=[a_{n\eps}]=[\bar{a}_{n\eps}]$ for all $n\in\N$
are two sequences which are moderate over hypersums, does the equality
$\left[\sum_{n=0}^{\nint{(N)}_{\eps}}a_{n\eps}\right]=\left[\sum_{n=0}^{\nint{(N)}_{\eps}}\bar{a}_{n\eps}\right]$
hold? The answer is negative: let $a_{n\eps}:=0$ if $\eps<\frac{1}{n+1}$
and $a_{n\eps}:=1$ otherwise, then $[a_{n\eps}]=0$ are representatives
of zero, but the corresponding series is not zero:
\begin{equation}
\forall N=[N_{\eps}]\in\hyperN{\rho}:\ \sum_{n=0}^{N_{\eps}}a_{n\eps}=\sum_{n=\left\lceil \frac{1}{\eps}\right\rceil -1}^{N_{\eps}}1=N_{\eps}-\left\lceil \frac{1}{\eps}\right\rceil +2.\label{eq:reprCountEx2}
\end{equation}
Note that both examples \eqref{eq:inftySeries} and \eqref{eq:reprCountEx2}
show that in dealing with hypersums, also the values $a_{n\eps}$
for ``$\eps$ large'' may play a role. We can then proceed like
for \eqref{eq:moderateHypseries} by saying that $(a_{n\eps})_{n,\eps}\sim_{\sigma\rho}(\bar{a}_{n\eps})_{n,\eps}$
if
\begin{equation}
\forall M,N\in\hyperN{\sigma}:\ \left(\sum_{n=\nint{(N)}_{\eps}}^{\nint{(M)}_{\eps}}\left(a_{n\eps}-\bar{a}_{n\eps}\right)\right)\sim_{\rho}0.\label{eq:negligibleHyperseries}
\end{equation}

The idea of this second solution is hence to consider hyperseries
only for representatives which are moderate over hypersums modulo
the equivalence relation \eqref{eq:negligibleHyperseries}.

In the following definition, we will consider both solutions:
\begin{defn}
\label{def:RCringHyperseries}The quotient set $\left(\R^{\N\times I}\right)_{\sigma\rho}/\sim_{\sigma\rho}\,=:\rcrhos$
of the set $\left(\R^{\N\times I}\right)_{\sigma\rho}$ of nets which
are $\sigma$, $\rho$-\emph{moderate over hypersums}
\[
(a_{n\eps})_{n,\eps}\in\left(\R^{N\times I}\right)_{\sigma\rho}\ :\iff\ \forall N\in\hyperN{\sigma}:\ \left(\sum_{n=0}^{\nint{(N)}_{\eps}}a_{n\eps}\right)\in\R_{\rho},
\]
by \emph{$\sigma$, $\rho$-negligible nets
\[
(a_{n\eps})_{n,\eps}\sim_{\sigma\rho}(\bar{a}_{n\eps})_{n,\eps}\ :\iff\ \forall M,N\in\hyperN{\sigma}:\ \left(\sum_{n=\nint{(N)}_{\eps}}^{\nint{(M)}_{\eps}}\left(a_{n\eps}-\bar{a}_{n\eps}\right)\right)\sim_{\rho}0,
\]
}is called the \emph{space of sequences for hyperseries}. Nets of
$\R^{\N\times I}$ are denoted as $(a_{n\eps})_{n,\eps}$ or simply
as $(a_{n\eps})$; equivalence classes of $\rcrhos$ are denoted as
$(a_{n})_{n}=[(a_{n\eps})_{n,\eps}]=[a_{n\eps}]\in\rcrhos$.\\
The ring of Robinson-Colombeau defined by the index set $\N\times I$
ordered by
\begin{equation}
(n,\eps)\le(m,e)\ :\iff\eps\le e,\label{eq:orderNtimesI}
\end{equation}
is denoted as $\rcrhou$. We recall that $\rcrhou$ is the quotient
space of moderate nets $(a_{n\eps})_{n,\eps}$ (i.e.~\eqref{eq:uniformlyModerate}
holds) modulo negligible nets (i.e.~\eqref{eq:uniformlyNeglibible}
holds) with respect to the directed order relation \eqref{eq:orderNtimesI}.
\end{defn}

\noindent The letter 'u' in $\rcrhou$ recalls that in this case we
are considering \emph{uniformly }moderate (and \emph{uniformly }negligible)
sequences. The letter 's' in $\rcrhos$ recalls that this is the space
of \emph{sequences }for hyperseries. We recall that $\rcrhou$ is
a ring with respect to pointwise multiplication $(a_{n})_{n}\cdot(b_{n})_{n}=\left[\left(a_{n\eps}\cdot b_{n\eps}\right)_{n\eps}\right]$.
When we want to distinguish equivalence classes in these two quotient
sets, we use the notations
\[
(a_{n})_{n}=\left[a_{n\eps}\right]_{\text{s}}\in\rcrhos,\quad\left\{ a_{n}\right\} _{n}=\left[a_{n\eps}\right]_{\text{u}}\in\rcrhou.
\]
Note explicitly that, on the contrary with respect to $\rcrhos$,
the ring $\rcrhou$ depends on only one gauge $\rho$.

We already observed the importance to consider suitable relations
between the two gauges $\sigma$ and $\rho$:
\begin{defn}
\label{def:*rel}Let $\sigma$, $\rho$ be two gauges, then we define
\begin{align*}
\sigma & \ge\rho^{*}\ :\iff\ \exists Q_{\sigma,\rho}\in\N\,\forall^{0}\eps:\ \sigma_{\eps}\ge\rho_{\eps}^{Q_{\sigma,\rho}}.\\
\sigma & \le\rho^{*}\ :\iff\ \exists Q_{\sigma,\rho}\in\N_{>0}\,\forall^{0}\eps:\ \sigma_{\eps}\le\rho_{\eps}^{Q_{\sigma,\rho}}.
\end{align*}
(Note that if $\sigma\ge\rho^{*}$, then $Q_{\sigma,\rho}>0$ since
$\sigma_{\eps}\to0$).
\end{defn}

\noindent It is easy to prove that $(-)\le(-)^{*}$ is a reflexive,
transitive and antisymmetric relation, in the sense that $\sigma\le\rho^{*}$
and $\rho\le\sigma^{*}$ imply $\sigma_{\eps}=\rho_{\eps}$ for $\eps$
small, whereas $(-)\ge(-)^{*}$ is only a partial order. Clearly,
$\sigma\ge\rho^{*}$ is equivalent to the inclusion of $\sigma$-moderate
nets $\R_{\sigma}\subseteq\R_{\rho}$, whereas $\sigma\le\rho^{*}$
is equivalent to $\R_{\sigma}\supseteq\R_{\rho}$.

It is well-known that there is no natural product between two ordinary
series in $\R$ and involving summations with only one index set in
$\N$ (see e.g.~\cite{BoKh06} and Sec.~\ref{sec:Cauchy-product}
below). This is the main motivation to consider only a structure of
$\rcrho$-module on $\rcrhos$ and, later, the natural Cauchy product
between hyperseries:
\begin{thm}
\label{thm:module}$\rcrhos$ is a quotient $\rcrho$-module.
\end{thm}

\begin{proof}
The closure of the set of $\sigma$, $\rho$-moderate nets over hypersums,
i.e.~$(\R^{\N\times I})_{\sigma\rho}=\left\{ (a_{n\eps})\mid\forall N\in\hyperN{\sigma}:\ \left(\sum_{n=0}^{\nint{(N)}_{\eps}}a_{n\eps}\right)\in\R_{\rho}\right\} $
with respect to pointwise sum $(a_{n\eps}+b_{n\eps})$ and product
$(r_{\eps})\cdot(a_{n,\eps})=(r_{\eps}\cdot a_{n,\eps})$ by $(r_{\eps})\in\R_{\rho}$
follows from similar properties of the ring $\R_{\rho}$. Similarly
to the case of $\rcrho$, we can finally prove that the equivalence
relation \eqref{eq:negligibleHyperseries} is a congruence with respect
to these operations.
\end{proof}
\noindent We now prove that if $(a_{n})_{n}\in\rcrhos$, then hyperfinite
sums are well-defined:
\begin{thm}
\label{thm:seriesWellDef}Let $(a_{n})_{n}=[a_{n\eps}]\in\rcrhos$
and let $\sigma$, $\rho$ be two arbitrary gauges, then the map
\[
(M,N)\in\hyperN{\sigma}^{2}\mapsto\left[\sum_{n=\nint{(N)}_{\eps}}^{\nint{(M)}_{\eps}}a_{n\eps}\right]\in\rcrho
\]
is well-defined.
\end{thm}

\begin{proof}
$\rho$-moderateness directly follows from $(a_{n})_{n}\in\rcrhos$;
in fact:
\begin{equation}
\left|\sum_{n=\nint{(N)}_{\eps}}^{\nint{(M)}_{\eps}}a_{n\eps}\right|=\left|\sum_{n=0}^{\nint{(M)}_{\eps}}a_{n\eps}-\sum_{n=0}^{\nint{(N)}_{\eps}-1}a_{n\eps}\right|\le\rho_{\eps}^{-Q_{1}}+\rho_{\eps}^{-Q_{2}}.\label{eq:partSumsModerate}
\end{equation}
Now, assume that $(a_{n})_{n}=[\bar{a}_{n\eps}]$ is another representative.
For $q\in\N$, condition \eqref{eq:negligibleHyperseries} yields
\begin{align*}
\left|\sum_{n=\nint{(N)}_{\eps}}^{\nint{(M)}_{\eps}}a_{n\eps}-\sum_{n=\nint{(N)}_{\eps}}^{\nint{(M)}_{\eps}}\bar{a}_{n\eps}\right| & \le\rho_{\eps}^{q}
\end{align*}
for $\eps$ small.
\end{proof}
\noindent We can finally define hyperfinite sums and hyperseries.
\begin{defn}
\label{def:hyperseries} Let $(a_{n})_{n}=[a_{n\eps}]\in\rcrhos$
and let $\sigma$ and $\rho$ be two arbitrary gauges then the term
\begin{equation}
\sum_{n=N}^{M}a_{n}:=\left[\sum_{n=\nint{(N)}_{\eps}}^{\nint{(M)}_{\eps}}a_{n\eps}\right]\in\rcrho\quad\forall N,M\in\hyperN{\sigma}\label{eq:hyperfiniteSum}
\end{equation}
is called $\sigma$, $\rho$-\emph{hypersum} of $(a_{n})_{n}$ (\emph{hypersum}
for brevity).

\noindent Moreover, we say that $s$ \emph{is the $\rho$-sum of hyperseries
with terms} $(a_{n})_{n\in\N}$\emph{ as $n\in\hypNs$} if $s$ is
the hyperlimit of the hypersequence $N\in\hyperN{\sigma}\mapsto\sum_{n=0}^{N}a_{n}\in\rcrho$.
In this case, we write

\[
s=\hyperlimarg{\rho}{\sigma}{N}\sum_{n=0}^{N}a_{n}=:\hypersum{\rho}{\sigma}a_{n}.
\]
In case the use of the gauge $\rho$ is clear from the context, we
simply say that $s$ \emph{is the sum of the hyperseries with terms}
$(a_{n})_{n\in\N}$ \emph{as} $n\in\hypNs$.\\
As usual, we also say that the hyperseries $\hypersum{\rho}{\sigma}a_{n}$
is \emph{convergent} if 
\[
\exists\,\hyperlimarg{\rho}{\sigma}{N}\sum_{n=0}^{N}a_{n}\in\rcrho.
\]

\noindent Whereas, we say that a hyperseries $\hypersum{\rho}{\sigma}a_{n}$
\emph{does not converge} if $\hyperlimarg{\rho}{\sigma}{N}\sum_{n=0}^{N}a_{n}$
does not exist in $\rcrho$. More specifically, if $\hyperlimarg{\rho}{\sigma}{N}\sum_{n=0}^{N}a_{n}=+\infty$
($-\infty$), we say that $\hypersum{\rho}{\sigma}a_{n}$ \emph{diverges
to $+\infty$ }(\emph{$-\infty$}).
\end{defn}

\noindent For the sake of brevity, when dealing with hyperseries or
with hypersums, we always \emph{implicitly assume} that $\sigma$,
$\rho$ are two gauges and that $(a_{n})_{n}\in\rcrhos$.
\begin{rem}
\label{rem:hypfiniteSums}~
\begin{enumerate}
\item Note that we are using an abuse of notations, since the term $\sum_{n=N}^{M}a_{n}$
actually depends on the two considered gauges $\rho$, $\sigma$.
\item Explicitly, $s=\hypersum{\rho}{\sigma}a_{n}$ means
\begin{equation}
\forall q\in\N\,\exists M\in\hyperN{\sigma}\,\forall N\in\hyperN{\sigma}_{\ge M}:\ \left|\sum_{n=0}^{N}a_{n}-s\right|<\diff{\rho}^{q}.\label{eq:hyperLimConvHyperSer}
\end{equation}
\item Also note that if $N$, $M\in\N$, then $\sum_{n=N}^{M}a_{n}=a_{N}+\ldots+a_{M}\in\rti$
is the usual finite sum. This is related to our motivating discussion
about the definition of hyperfinite series at the beginning of Sec.~\ref{sec:Hyperseries}.
\end{enumerate}
\end{rem}

\subsection{Relations between $\rcrhos$ and $\rcrhou$}

A first consequence of the condition $\sigma\ge\rho^{*}$ is that
if the net $(a_{n\eps})$ is uniformly moderate, then it is also moderate
over hypersums. Similarly, we can argue for the equality, so that
we have a natural map $\rcrhou\ra\rcrhos$:
\begin{lem}
\label{lem:Rtilu-RtilsComparison}We have the following properties:
\begin{enumerate}
\item \label{enu:extension}If $N=[N_{\eps}]\in\hyperN{\sigma}$, with $N_{\eps}\in\N$,
and $(a_{n})_{n}=[a_{n\eps}]\in\rcrhos$, then $a_{N}:=[a_{N_{\eps},\eps}]\in\rcrho$
is well defined. That is, any sequence $(a_{n})_{n}\in\rcrhos$ can
be extended from $\N$ to the entire set $\hyperN{\sigma}$ of hyperfinite
numbers.
\end{enumerate}
Moreover, if we assume that $\sigma\ge\rho^{*}$, then we have:
\begin{enumerate}[resume]
\item \label{enu:embedding}Both $\rcrho\subseteq\rcrhou$ and $\rcrho\subseteq\rcrhos$
via the embedding $[x_{\eps}]\mapsto[(x_{\eps})_{n,\eps}]$.
\item \label{enu:emb2}The mapping
\[
\lambda:\left[a_{n\eps}\right]_{\text{\emph{u}}}\in\rcrhou\mapsto\left[a_{n\eps}\right]_{\text{\emph{s}}}\in\rcrhos
\]
is a well-defined $\rcrho$-linear map.
\item \label{enu:convRtilu}Let us define a hypersum operator
\[
\sum_{n=N}^{M}:\rcrhou\ra\rcrho\quad\text{as}\quad\sum_{n=N}^{M}\left[a_{n\eps}\right]_{\text{\emph{u}}}:=\left[\sum_{n=\nint{(N)}_{\eps}}^{\nint{(M)}_{\eps}}a_{n\eps}\right]
\]
for all $M$, $N\in\hypNs$. Then $\sum_{n=N}^{M}\left[a_{n\eps}\right]_{\text{\emph{u}}}=\sum_{n=N}^{M}\lambda\left(\left[a_{n\eps}\right]_{\text{\emph{u}}}\right)\in\rcrho$.
Therefore, the character (convergent or divergent) of the corresponding
hyperseries is identical in the two spaces $\rcrhou$ and $\rcrhos$.
\end{enumerate}
\end{lem}

\begin{proof}
\ref{enu:extension}: In fact, if $(N_{\eps})\in\N_{\sigma}$, then
from $(a_{n})_{n}\in\rcrhos$ we get the existence of $Q_{i}\in\N$
such that
\[
\left|a_{N_{\eps},\eps}\right|=\left|\sum_{n=0}^{N_{\eps}}a_{n\eps}-\sum_{n=0}^{N_{\eps}-1}a_{n\eps}\right|\le\rho_{\eps}^{-Q_{1}}+\rho_{\eps}^{-Q_{2}}.
\]
Finally, if $(a_{n})_{n}=[a_{n\eps}]=[\bar{a}_{n\eps}]\in\rcrhos$,
then directly from \eqref{eq:negligibleHyperseries} with $M=N$ we
get $[(a_{\nint{(N)}_{\eps},\eps})_{n,\eps}]=[(\bar{a}_{\nint{(N)}_{\eps},\eps})_{n,\eps}]$
(note that it is to have this result that we defined \eqref{eq:negligibleHyperseries}
using $\sum_{\nint{(N)}_{\eps}}^{\nint{(M)}_{\eps}}$ instead of $\sum_{n=0}^{\nint{(N)}_{\eps}}$
like in \eqref{eq:moderateHypseries}). Therefore, $a_{N}:=[a_{N_{\eps},\eps}]\in\rcrho$
is well-defined. In particular, this applies with $N=n\in\N$, so
that any equivalence class $(a_{n})_{n}=[(a_{n\eps})_{n,\eps}]\in\rcrhos$
also defines an ordinary sequence $(a_{n})_{n\in\N}=\left(\left[a_{n\eps}\right]\right)_{n\in\N}$
of $\rcrho$. On the other hand, let us explicitly note that if $(a_{n})_{n\in\N}=(\bar{a}_{n})_{n\in\N}$,
i.e.~if $a_{n}=\bar{a}_{n}$ for all $n\in\N$, then not necessarily
\eqref{eq:negligibleHyperseries} holds, i.e.~we can have $(a_{n})_{n}\ne(\bar{a}_{n})_{n}$
as elements of the quotient module $\rcrhos$.

Claim \ref{enu:embedding} is left to the reader.

\ref{enu:emb2}: Assume that inequality in \eqref{eq:uniformlyModerate}
holds for $\eps\le\eps_{0}$, i.e.
\begin{equation}
\forall\eps\le\eps_{0}\,\forall n\in\N:\ \left|a_{n\eps}\right|\le\rho_{\eps}^{-Q}.\label{eq:unifModEps0}
\end{equation}
Without loss of generality, we can also assume that $\nint{(N)}_{\eps}+1\le\sigma_{\eps}^{-R}$
and $\sigma_{\eps}\ge\rho_{\eps}^{-Q_{\sigma,\rho}}$. Then, for each
$\eps\le\eps_{0}$, using \eqref{eq:unifModEps0} we have $\left|\sum_{n=0}^{\nint{(N)}_{\eps}}a_{n\eps}\right|\le\left(\nint{(N)}_{\eps}+1\right)\cdot\rho_{\eps}^{-Q}\le\sigma_{\eps}^{-R}\cdot\rho_{\eps}^{-Q}\le\rho_{\eps}^{-R\cdot Q_{\sigma,\rho}-Q}$.
Moreover, if $\left[a_{n\eps}\right]_{\text{u}}=0$ then, for all
$q\in\N$
\[
\left|\sum_{n=0}^{\nint{(N)}_{\eps}}a_{n\eps}\right|\le\left(\nint{(N)}_{\eps}+1\right)\cdot\rho_{\eps}^{q}
\]
for $\eps$ small. Since $\nint{(N)}_{\eps}$ is $\sigma$-moderate
and $\sigma\geq\rho^{*}$, this proves that the linear map $\lambda$
is well-defined.

\ref{enu:convRtilu}: This follows directly from Def.~\ref{def:hyperseries}.
\end{proof}
\noindent It remains an open problem the study of injectivity of the
map $\lambda$.

We could say that $\sigma\ge\rho^{*}$ and $\left[a_{n\eps}\right]_{\text{u}}\in\rcrhou$
are sufficient conditions to get $(a_{n})_{n}=\left[a_{n\eps}\right]_{\text{s}}\in\rcrhos$
and hence to start talking about hyperfinite sums $\sum_{n=N}^{M}a_{n}$
and hyperseries $\hypersum{\rho}{\sigma}a_{n}$. On the other hand,
if $\sigma\ge\rho^{*}$ is false, we can still consider the space
$\rcrhos$ and hence talk about hyperseries, but the corresponding
space $\rcrhos$ lacks elements such as $(1)_{n}$ because of the
subsequent Lem.~\ref{lem:1}. As we will see in the following examples
\ref{exa:GeomExpSeries}.\ref{enu:geomSigma} and \ref{exa:GeomExpSeries}.\ref{enu:expSigma},
the space $\rcrhos$ still contains sequences corresponding to interesting
converging hyperseries.

The following examples of convergent hyperseries justify our definition
of hyperseries by recovering classical examples such such geometric
and exponential hyperseries. We recall that this is not possible using
classical series in a non-Archimedean setting.
\begin{example}
\label{exa:GeomExpSeries}\ 
\begin{enumerate}[label=\arabic*) ]
\item Let $N=[N_{\eps}]\in\hypNs$, where $N_{\eps}\in\N$ for all $\eps$,
then $\sum_{n=N}^{N}a_{n}=\left[a_{N_{\eps},\eps}\right]=a_{N}\in\rcrho$.
We recall that $a_{N}=\left[a_{N_{\eps},\eps}\right]$ is the extension
of $(a_{n})_{n}\in\rcrhos$ to $N\in\hypNs$ (see \ref{enu:extension}
of Lem.~\ref{lem:Rtilu-RtilsComparison}).
\item \label{enu:geomSeries}For all $k\in\rcrho$, $0<k<1$, we have (note
that $\sigma=\rho$)
\begin{equation}
\hypersum{\rho}{\rho}k^{n}=\frac{1}{1-k}.\label{eq:geomHyperSeries}
\end{equation}
We first note that $k^{n}\le1$, so that $\left\{ k^{n}\right\} _{n}=\left[k_{\eps}^{n}\right]_{\text{u}}\in\rcrhou$.
Now,
\[
\hypersum{\rho}{\rho}k^{n}=\hyperlimarg{\rho}{\rho}{N}\sum_{n=0}^{N}k^{n}=\hyperlimarg{\rho}{\rho}{N}\frac{1-k^{N+1}}{1-k}.
\]
But $k^{N+1}<\diff{\rho}^{q}$ if and only if $(N+1)\log k<q\log\diff{\rho}$.
Since $0<k<1$, we have $\log k<0$ and we obtain $N>q\frac{\log\diff{\rho}}{\log k}-1$.
It suffices to take $M_{\eps}:=\text{int}\left(q\frac{\log\rho_{\eps}}{\log k_{\eps}}\right)$
in the definition of hyperseries.
\item \label{enu:geomSigma}More generally, if $k\in\rcrho$, $0<k<1$,
we can evaluate
\[
\left[\left|\sum_{n=0}^{N_{\eps}}k_{\eps}^{n}\right|\right]=\left[\left|\frac{1-k_{\eps}^{N_{\eps}+1}}{1-k_{\eps}}\right|\right]\le\frac{2}{1-k}\in\rcrho.
\]
This shows that $(k^{n})_{n}=[k_{\eps}^{n}]\in\rcrhos$ for all gauges
$\sigma$. If we assume $\sigma\le\rho^{*}$ then $M_{\eps}:=\text{int}\left(q\frac{\log\rho_{\eps}}{\log k_{\eps}}\right)\in\N_{\sigma}$
and hence, proceeding as above, we can prove that $\hypersum{\rho}{\sigma}k^{n}=\frac{1}{1-k}$.
Note, e.g., that if $\sigma_{\eps}:=\log\left(-\log\rho_{\eps}\right)^{-1}$,
then $\sigma\not\le\rho^{*}$ and $M_{\eps}=\text{int}\left(q\frac{\log\rho_{\eps}}{\log k_{\eps}}\right)\notin\N_{\sigma}$.
\item Let $k\in\rcrho_{>0}$ be such that $k\sbpt{>}1$ (see \cite{MTAG}
for the relations $\sbpt{>}$ and $\sbpt{=}$ and, more generally,
for the language of subpoints), then the hyperseries $\hypersum{\rho}{\rho}k^{n}$
is not convergent. In fact, by contradiction, in the opposite case
we would have $\sum_{n=0}^{N}k^{n}\in(l-1,l+1)$ for some $l\in\rti$
for all $N$ sufficiently large, but this is impossible because for
all fixed $K\in\rcrho_{>0}$ and for $N$ sufficiently large, $\sum_{n=0}^{N}k^{n}=\hyperlimarg{\rho}{\rho}{N}\frac{1-k^{N+1}}{1-k}\sbpt{>}K$
because $\hyperlimarg{\rho}{\rho}{N}k^{N+1}\sbpt{=}+\infty$.
\item \label{enu:exp}For all $x\in\rcrho$ finite, we have $\hypersum{\rho}{\rho}\frac{x^{n}}{n!}=e^{x}$.
We have $|x|<M\in\R_{>0}$ because $x$ is finite, and hence $\left|x_{\eps}\right|\le M$
for all $\eps\le\eps_{0}$. Thereby $\frac{x_{\eps}^{n}}{n!}\le\frac{\left|x_{\eps}\right|^{n}}{n!}\le e^{M}$
for all $n\in\N$ and thus $\left\{ \frac{x^{n}}{n!}\right\} _{n}=\left[\frac{x_{\eps}^{n}}{n!}\right]_{\text{u}}\in\rcrhou$.
For all $N=[N_{\eps}]\in\hyperN{\sigma}$, $N_{\eps}\in\N$, and all
$\eps$, we have 
\begin{equation}
\sum_{n=0}^{N_{\eps}}\frac{x_{\eps}^{n}}{n!}=e^{x_{\eps}}-\sum_{n=N_{\eps}+1}^{+\infty}\frac{x_{\eps}^{n}}{n!}.\label{eq:expSeries}
\end{equation}
Now, take $N\in\hyperN{\sigma}$ such that $\frac{M}{N+1}<\frac{1}{2}$,
so that we can assume $N_{\eps}+1>2M$ for all $\eps$. We have $\left|\sum_{n=N_{\eps}+1}^{+\infty}\frac{x_{\eps}^{n}}{n!}\right|\le\sum_{n>N_{\eps}}\frac{M^{n}}{n!}$,
and for all $n\ge N_{\eps}$ we have (by induction)
\[
\frac{M^{n+1}}{(n+1)!}<\frac{1}{2^{n+1}}.
\]
Therefore $\left|\sum_{n=N_{\eps}+1}^{+\infty}\frac{x_{\eps}^{n}}{n!}\right|\le\sum_{n>N_{\eps}}\frac{1}{2^{n}}$
and hence $\hyperlimarg{\rho}{\rho}{N}\sum_{n=N+1}^{+\infty}\frac{x^{n}}{n!}=0$
by \eqref{eq:geomHyperSeries}. This and \eqref{eq:expSeries} yields
the conclusion.
\item \label{enu:expSigma}In the same assumptions of the previous example,
we have $\left[\left|\sum_{n=0}^{N_{\eps}}\frac{x_{\eps}^{n}}{n!}\right|\right]\le e^{M}$
and hence $\left(\frac{x^{n}}{n!}\right)_{n}=\left[\frac{x_{\eps}^{n}}{n!}\right]\in\rcrhos$
for all gauges $\sigma$. If $\sigma\le\rho^{*}$, proceeding as above,
we can prove that $\hypersum{\rho}{\sigma}\frac{x^{n}}{n!}=e^{x}$.
\end{enumerate}
\end{example}

Finally, the following lemma shows that a sharply bounded sequence
of $\rcrho$ always defines a sequence for hyperseries, i.e.~an element
of $\rcrhos$.
\begin{lem}
\label{Lem:fromSeqToRcrhos}Let $(a_{n})_{n\in\N}$ be a sequence
of $\rcrho$. If $(a_{n})_{n\in\N}$ is sharply bounded:
\begin{equation}
\exists M\in\rcrho_{>0}\,\forall n\in\N:\ |a_{n}|\le M,\label{eq:a_nBounded}
\end{equation}
then there exists a sequence $(a_{n\eps})_{n\in\N}$ of $\R_{\rho}$
such that
\begin{enumerate}
\item $a_{n}=[a_{n\eps}]\in\rcrho$ for all $n\in\N$;
\item $[a_{n\eps}]_{\text{\emph{u}}}\in\rcrhou$;
\item If $\sigma\ge\rho^{*}$, then $[a_{n\eps}]_{\text{\emph{s}}}\in\rcrhos$.
\end{enumerate}
\end{lem}

\begin{proof}
Let $M=[M_{\eps}]$ be any representative of the bound satisfying
\eqref{eq:a_nBounded}, so that $M_{\eps}\le\rho_{\eps}^{-Q}$ for
$\eps\le\eps_{0}$ and for some $Q\in\N$. From \eqref{eq:a_nBounded},
for each $n\in\N$ we get the existence of a representative $a_{n}=[\bar{a}_{n\eps}]$
such that $|\bar{a}_{n\eps}|\le M_{\eps}$ for $\eps\le\eps_{0n}\le\eps_{0}$.
It suffices to define $a_{n\eps}:=\bar{a}_{n\eps}$ if $\eps\le\eps_{0n}$
and $a_{n\eps}:=M_{\eps}$ otherwise to have $\forall\eps\le\eps_{0}\,\forall n\in\N:\ \left|a_{n\eps}\right|\le\rho_{\eps}^{-Q}$,
so that $[a_{n\eps}]_{\text{u}}\in\rcrhou$. If $\sigma\ge\rho^{*}$,
we can then apply Lem.~\ref{lem:Rtilu-RtilsComparison}.
\end{proof}
However, let us note that, generally speaking, changing representatives
of $a_{n}$ as in the previous proof, we also get a different value
of the corresponding hyperseries, as proved by example \eqref{eq:reprCountEx2}.

\subsection{Divergent hyperseries}

By analyzing when the constant net $(1)$ is moderate over hypersums,
we discover the relation \eqref{eq:sigma-rho-min} between the gauges
$\sigma$ and $\rho$:
\begin{lem}
\label{lem:1}The constant net $(1)\in(\R^{\N\times I})_{\sigma\rho}$,
i.e.~it is $\sigma$, $\rho$-moderate over hypersums, if and only
if $\sigma\ge\rho^{*}$.
\end{lem}

\begin{proof}
If $(1)\in(\R^{\N\times I})_{\sigma\rho}$, we set $N_{\eps}:=\text{int}(\sigma_{\eps}^{-1})+1$,
so that we get $\sigma_{\eps}^{-1}\le\sum_{n=0}^{N_{\eps}}1=N_{\eps}\le\rho_{\eps}^{-Q_{\sigma,\rho}}$
for some $Q_{\sigma,\rho}\in\N$, i.e.~$\sigma\ge\rho^{*}$. Vice
versa, if $\sigma_{\eps}\ge\rho_{\eps}^{Q_{\sigma,\rho}}$ for all
$\eps\le\eps_{0}$, then for those $\eps$, we have $\left|\sum_{n=0}^{\nint{(N)}_{\eps}}1\right|=\nint{(N)}_{\eps}\le\sigma_{\eps}^{-R}\le\rho_{\eps}^{-R\cdot Q_{\sigma,\rho}}$
for some $R\in\N$, because $N\in\hypNs$.
\end{proof}
One could argue that we are mainly interested in converging hyperseries
and hence it is not worth considering the constant net $(1)$. On
the other hand, we would like to argue in the following way: the hypersums
$N\in\hypNs\mapsto\sum_{n=0}^{N}1\in\rcrho$ can be considered, but
they do not converge because $1\not\to0$. As we will see in Lem.~\ref{lem:NtoN+M},
this argumentation is possible only if $\sigma\ge\rho^{*}$ because
of the previous Lem\@.~\ref{lem:1}.

Let $\sigma\ge\rho^{*}$ and $\omega\in\hypNr$ be an infinite number.
If $(a_{n})_{n}\in\rcrhos$, we can also think at $\sum_{n=0}^{\omega}a_{n}$
as another way to compute an infinite summation of the numbers $a_{n}$.
In other words, the following examples can be considered as related
to calculation of divergent series. They strongly motivate and clarify
our definition of hyperfinite sum.
\begin{example}
\label{exa:div}~Assuming $\sigma\ge\rho^{*}$ and working with the
module $\rcrhos$, we have:
\begin{enumerate}[label=\arabic*) ]
\item [5)]\setcounter{enumi}{5}$\sum_{n=1}^{\omega}1=\omega\in\rcrho\setminus\R$.
\item \label{enu:div1}$\sum_{n=1}^{\omega}n=\left[\sum_{n=1}^{\nint{(\omega)}_{\eps}}n_{\eps}\right]=\left[\frac{\nint{(\omega)}_{\eps}(\nint{(\omega)}_{\eps}+1)}{2}\right]=\frac{\omega(\omega+1)}{2}$.
\item $\sum_{n=1}^{\omega}(2n-1)=2\sum_{n=1}^{\omega}n-\sum_{n=1}^{\omega}1=\omega^{2}$
because we know from Thm. \ref{thm:module} that $\rcrhos$ is an
$\rcrho$-module.
\item $\sum_{n=1}^{\omega}\left(a+(n-1)d\right)=\omega a+\frac{\omega^{2}d}{2}-\frac{\omega d}{2}$.
\item Using $\eps$-wise calculations, we also have $\sum_{n=0}^{\omega}(-1)^{n}=\frac{1}{2}(-1)^{\omega+1}+\frac{1}{2}$.
Note that the final result is a finite generalized number of $\rcrho$,
but it does not converge for $\omega\to+\infty$, $\omega\in\hypNr$.
\item \label{enu:div6}The net $(2^{n})_{n,\eps}$ is not $\rho$-moderate
over hypersums, in fact if $\omega=\left[\text{int}(\rho_{\eps}^{-1})\right]$,
then $\sum_{n=0}^{\omega_{\eps}}2^{n-1}=2^{\omega_{\eps}}-1$, which
is not $\rho$-moderate. Another possibility to consider the function
$2^{\omega}$ is to take another gauge $\mu\le\rho$ and the subring
of $\RC{\mu}$ defined by
\[
\RCud{\mu}{\rho}:=\{x\in\RC{\mu}\mid\exists N\in\N:\ |x|\le\diff{\rho}^{-N}\},
\]
where only here we have set $\diff{\rho}:=[\rho_{\eps}]_{\sim_{\mu}}\in\RC{\mu}$.
If we have
\begin{equation}
\forall N\in\N\,\exists M\in\N:\ \diff{\rho}^{-N}\le-M\log\diff{\mu},\label{eq:expTwoGauges}
\end{equation}
then $2^{(-)}:[x_{\eps}]\in\RCud{\mu}{\rho}\mapsto\left[2^{x_{\eps}}\right]\in\RC{\mu}$
is well defined. For example, if $\mu_{\eps}:=\exp\left(-\rho_{\eps}^{1/\eps}\right)$,
then $\mu\le\rho$ and \eqref{eq:expTwoGauges} holds for $M=1$.
Note that the natural ring morphism $[x_{\eps}]_{\sim_{\sigma}}\in\RCud{\mu}{\rho}\mapsto[x_{\eps}]_{\sim_{\rho}}\in\rcrho$
is surjective but generally not injective. Now, if $\omega\in\RCud{\mu}{\rho}\cap\hypNs$,
then $\sum_{n=0}^{\omega}2^{n-1}=2^{\omega}-1\in\RC{\mu}\setminus\RCud{\mu}{\rho}$.
\item \label{enu:Binomial formula}Once again, proceeding by $\eps$-wise
calculations, we also have the binomial formula: For all $a$, $b\in\rcrho$
and $n\in\hyperN{\rho}$, if $\left(a+b\right)^{n}\in\rcrho$, then
\[
\left(a+b\right)^{n}=\sum_{k=0}^{n}\left(\begin{array}{c}
n\\
k
\end{array}\right)a^{k}b^{n-k}
\]
where $n!:=\left[\text{\emph{ni}}(n)_{\eps}!\right]$ and $\left(\begin{array}{c}
n\\
k
\end{array}\right):=\frac{n!}{k!\left(n-k\right)!}$.
\end{enumerate}
\end{example}

\subsection{$\eps$-wise convergence and hyperseries}

\noindent The following result allows us to obtain hyperseries by
considering $\eps$-wise convergence of its summands. Its proof is
clearly very similar to that of \cite[Thm.~28]{MTAG}, but with a
special attention to the condition $(a_{n})_{n}\in\rcrhos$ that we
need beforehand to talk about hyperseries.
\begin{thm}
\label{thm:epsConv}Let $a:\N\ra\rcrho$ and $a_{n}=[a_{n\eps}]$
for all $n\in\N$. Let $q_{\eps}$, $M_{\eps}\in\N_{>0}$ be such
that $(q_{\eps})\to+\infty$ as $\eps\to0^{+}$, and
\begin{equation}
\forall\eps\,\forall N\ge M_{\eps}:\ s_{\eps}-\rho_{\eps}^{q_{\eps}}<\sum_{n=0}^{N}a_{n\eps}<s_{\eps}+\rho_{\eps}^{q_{\eps}}\label{eq:conv}
\end{equation}
(note that this implies that the standard series $\sum_{n=0}^{+\infty}a_{n\eps}$
converges to $s_{\eps}$). Finally, let $\mu$ be another gauge. If
$s=[s_{\eps}]\in\rti$, then setting $\sigma_{\eps}:=\min(\mu_{\eps},M_{\eps}^{-1})$,
we have:
\begin{enumerate}
\item $\sigma\le\mu$ is a gauge (not necessarily a monotonic one);
\item $M:=[M_{\eps}]\in\hypNs$;
\item $s_{M}:=\left[\sum_{n=0}^{M_{\eps}-1}a_{n\eps}\right]\in\rti$;
\item $(a_{n+M})_{n}:=\left[\left(a_{n+M_{\eps},\eps}\right)_{n,\eps}\right]_{\text{\emph{s}}}\in\rcrhos$;
\item $s=s_{M}+\hypersum{\rho}{\sigma}a_{n+M}$.
\end{enumerate}
\end{thm}

\begin{proof}
The net $\sigma_{\eps}=\min(\mu_{\eps},M_{\eps}^{-1})\to0^{+}$ as
$\eps\to0^{+}$ because $\mu_{\eps}\to0^{+}$, i.e.~it is a gauge
(note that not necessarily $\sigma$ is non-decreasing, e.g.~if $\lim_{\eps\to\frac{1}{k}}M_{\eps}=+\infty$
for all $k\in\N_{>0}$ and $M_{\eps}\ge\mu_{\eps}^{-1}$). We have
$M:=[M_{\eps}]\in\hypNs$ because our definition of $\sigma$ yields
$M_{\eps}\le\sigma_{\eps}^{-1}$. Moreover, since $(q_{\eps})\to+\infty$,
for all $\eps$ condition \eqref{eq:conv} yields $s_{\eps}-1\le s_{\eps}-\rho_{\eps}^{q_{\eps}}<\sum_{n=0}^{N}a_{n\eps}<s_{\eps}+\rho_{\eps}^{q_{\eps}}\le s_{\eps}+1$
for all $N\ge M_{\eps}$. For $N=M$ this gives $s_{M}=\left[\sum_{n=0}^{M_{\eps}}a_{n\eps}\right]\in\rti$,
because we assumed that $s=[s_{\eps}]\in\rcrho$. If $N\in\hypNs$,
then $\nint{(N)}_{\eps}+M_{\eps}\ge M_{\eps}$ for all $\eps$, and
hence $s_{\eps}-1<\sum_{n=0}^{M_{\eps}-1}a_{n\eps}+\sum_{n=M_{\eps}}^{\nint{(N)}_{\eps}+M_{\eps}}a_{n\eps}=s_{M\eps}+\sum_{n=0}^{\nint{(N)}_{\eps}}a_{n+M_{\eps},\eps}<s_{\eps}+1$.
This shows that $\left[\sum_{n=0}^{\nint{(N)}_{\eps}}a_{n+M_{\eps},\eps}\right]\in\rcrho$,
because $s_{M}$, $s\in\rcrho$, i.e.~$(a_{n+M})_{n}\in\rcrhos$.
Similarly, if $q\in\N$, then $q<q_{\eps}$ for $\eps$ sufficiently
small and, proceeding as above, we can prove that $|\sum_{n=0}^{N}a_{n+M}-s-s_{M}|<\diff{\rho}^{q}$.
\end{proof}
\noindent Note that, if $(M_{\eps})$ is not $\rho$-moderate and
we set $\sigma_{\eps}:=\min\left(\rho_{\eps},M_{\eps}^{-1}\right)\in(0,1]$,
then
\[
\forall Q\in\N\,\exists L\subzero I\,\forall\eps\in L:\ \sigma_{\eps}^{-1}>\bar{M}_{\eps}>\rho_{\eps}^{-Q}
\]
and this shows that $\sigma\ge\rho^{*}$, i.e.~the fundamental condition
to have moderateness of hyperfinite sums in the ring $\rcrhou$ (see
\eqref{eq:hyperfiniteSumR_u}), \emph{always necessarily does not
hold}. On the other hand, if $(M_{\eps})$ is $\rho$-moderate, then
$\R_{\sigma}=\R_{\mu}$; in particular, we can simply take $\sigma=\mu=\rho$.

Assuming absolute convergence, we can obtain a stronger and clearer
result:
\begin{thm}
\label{thm:epsConvAbs}Let $a:\N\ra\rcrho$ and $a_{n}=[a_{n\eps}]$
for all $n\in\N$. Assume that the standard series $\sum_{n=0}^{+\infty}a_{n\eps}$
converges absolutely to $\bar{s}_{\eps}$ and $[\bar{s}_{\eps}]\in\rti$.
Finally, let $\mu$ be another gauge and set $s_{\eps}:=\sum_{n=0}^{+\infty}a_{n\eps}$,
then there exists a gauge $\sigma\le\mu$ such that:
\begin{enumerate}[resume]
\item $(a_{n})_{n}=\left[a_{n\eps}\right]_{\text{\emph{s}}}\in\rcrhos$;
\item $s=\hypersum{\rho}{\sigma}a_{n}$.
\end{enumerate}
\end{thm}

\begin{proof}
Set $q_{\eps}:=\left\lceil \frac{1}{\eps}\right\rceil $, so that
the absolute convergence of $\sum_{n=0}^{+\infty}a_{n\eps}$ yields
\begin{equation}
\forall\eps\,\exists M_{\eps}\in\N_{>0}\,\forall N\ge M_{\eps}:\ \bar{s}_{\eps}-\rho_{\eps}^{q_{\eps}}<\sum_{n=0}^{N}\left|a_{n\eps}\right|<\bar{s}_{\eps}+\rho_{\eps}^{q_{\eps}}\label{eq:conv2}
\end{equation}
and we can thereby apply the previous Thm.~\ref{thm:epsConv} obtaining
$\bar{s}_{M}:=\left[\sum_{n=0}^{M_{\eps}-1}|a_{n\eps}|\right]$, $\hypersum{\rho}{\sigma}\left|a_{n+M}\right|\in\rcrho$.
Therefore, for an arbitrary $N\in\hypNs$, we have:
\begin{align*}
\left[\left|\sum_{n=0}^{\nint{(N)}_{\eps}}a_{n\eps}\right|\right] & \le\left[\sum_{n=0}^{\nint{(N)}_{\eps}}\left|a_{n\eps}\right|\right]\le\left[\sum_{n=0}^{M_{\eps}+\nint{(N)}_{\eps}}\left|a_{n\eps}\right|\right]=\\
 & =\bar{s}_{M}+\hypersum{\rho}{\sigma}\left|a_{n+M}\right|\in\rcrho,
\end{align*}
and this shows that $(a_{n})_{n}\in\rcrhos$. Now, proceeding as above
we can prove that $s=\hypersum{\rho}{\sigma}a_{n}$.
\end{proof}

\subsection{Basic properties of hyperfinite sums and hyperseries}

We now study some basic properties of hyperfinite sums \eqref{eq:hyperfiniteSum}.
\begin{lem}
\label{lem:N+M}Let $\sigma$, $\rho$ be arbitrary gauges, $(a_{n})_{n}\in\rcrhos$,
and $M$, $N\in\hypNs$, then
\begin{equation}
\sum_{n=0}^{N+M}a_{n}-\sum_{n=0}^{N}a_{n}=\sum_{n=N+1}^{N+M}a_{n}.\label{eq:N+M}
\end{equation}
\end{lem}

\begin{proof}
For simplicity, if $N=[N_{\eps}]$, $M=[M_{\eps}]$ with $N_{\eps},M_{\eps}\in\N$
for all $\eps$, then $\sum_{n=0}^{N}a_{n}=\left[\sum_{n=0}^{N_{\eps}}a_{n\eps}\right]\in\rcrho$
and
\begin{align*}
\sum_{n=0}^{N+M}a_{n} & =\left[\sum_{n=0}^{N_{\eps}+M_{\eps}}a_{n\eps}\right]=\\
 & =\left[\sum_{n=0}^{N_{\eps}}a_{n\eps}+\sum_{N_{\eps}+1}^{N_{\eps}+M_{\eps}}a_{n\eps}\right]=\\
 & =\sum_{n=0}^{N}a_{n}+\sum_{n=N+1}^{N+M}a_{n}.
\end{align*}
\end{proof}
\begin{lem}
\label{lem:NtoN+M}If $\hypersum{\rho}{\sigma}a_{n}$ is convergent
and $M\in\hypNs$, then
\[
\hypersum{\rho}{\sigma}a_{n}=\hyperlimarg{\rho}{\sigma}{N}\sum_{n=0}^{N+M}a_{n}.
\]
Therefore, from \eqref{eq:N+M}, we also have
\[
\hyperlimarg{\rho}{\sigma}{N}\sum_{n=N}^{N+M}a_{n}=0.
\]
In particular, \textup{$\hyperlim{\rho}{\sigma}a_{n}=0$.}
\end{lem}

\begin{proof}
Let $s:=\hypersum{\rho}{\sigma}a_{n}$. Directly from the definition
of convergent hyperseries \eqref{eq:hyperLimConvHyperSer}, we have

\begin{equation}
\forall q\in\N\,\exists K\in\hyperN{\sigma}\,\forall N\in\hyperN{\sigma}_{\ge K}:\ \left|\sum_{n=0}^{N}a_{n}-s\right|<\diff{\rho}^{q}.\label{eq:hyperLimConvHyperSer-1}
\end{equation}

\noindent In particular, if $N\in\hyperN{\sigma}_{\ge K}$, then also
$N+M\in\hyperN{\sigma}_{\ge K}$, which is our conclusion.
\end{proof}
Directly from Lem.~\ref{lem:N+M}, we also have:
\begin{cor}
\label{cor:convAddRem}Let $\hypersum{\rho}{\sigma}a_{n}$ be a convergent
hyperseries. Then adding a hyperfinite number of terms have no effect
on the convergence of the hyperseries, that is 
\[
\sum_{n=0}^{K}a_{n}+\hyperlimarg{\rho}{\sigma}{N}\sum_{n=K+1}^{N}a_{n}=\hypersum{\rho}{\sigma}a_{n}\quad\forall K\in\hypNs.
\]
\end{cor}

Mimicking the classical theory, we can also say that the hyperseries
$\hypersum{\rho}{\sigma}a_{n}$ is \emph{Cesàro hypersummable} if
$\exists\,\hyperlimarg{\rho}{\sigma}{N}\frac{1}{N+1}\sum_{n=0}^{N}\sum_{k=0}^{n}a_{k}$
(this hyperlimit is clearly called \emph{Cesàro hypersum}). For example,
proceeding $\eps$-wise, we have that the hyperseries $\hypersum{\rho}{\sigma}(-1)^{n}$
has Cesàro hypersum equal to $\frac{1}{2}$. Trivially generalizing
the usual proof, we can show that if $\hypersum{\rho}{\sigma}a_{n}$
converges to $s$, then it is also Cesàro hypersummable with the same
hypersum $s$.

\section{Hyperseries convergence tests}

\subsection{\label{subsec:pSeries}$p$-hyperseries}

To deal with $p$-hyperseries, i.e.~hyperseries of the form
\[
\parthypersumarg{\rho}{\sigma}{\hspace{1em}n}{\ge1}\frac{1}{n^{p}},
\]
we always assume that:
\begin{enumerate}
\item $p\in\rcrho$ is a \emph{finite} generalized number such that $p>0$
or $p<0$; we recall that if $p$ is an infinite number (at least
on a subpoint), the operation $n^{p}$ is in general not well-defined
since it leads to a non $\rho$-moderate net (see also~Example \ref{exa:div}.\ref{enu:div6}).
\item The gauges $\sigma$ and $\rho$ are chosen so that $\left(\frac{1}{n^{p}}\right)_{n}\in\rcrhos$.
A \emph{sufficient} condition for this is that $\sigma\ge\rho^{*}$,
i.e.~that $\R_{\sigma}\subseteq\R_{\rho}$. In fact, since $p$ is
finite, we have $\left(\sum_{n=1}^{N_{\eps}}n^{-p_{\eps}}\right)\le N_{\eps}N_{\eps}^{-p_{\eps}}\in\R_{\sigma}\subseteq\R_{\rho}$
whenever $N_{\eps}\in\N$, $\left(N_{\eps}\right)\in\N_{\sigma}$.
\end{enumerate}
Note explicitly that, from the assumption $\left(\frac{1}{n^{p}}\right)_{n}\in\rcrhos$,
we hence have
\begin{equation}
\sum_{n=1}^{N}\frac{1}{n^{p}}=\left[\sum_{n=1}^{N_{\eps}}\frac{1}{n^{p_{\eps}}}\right]\in\rcrho\label{eq:pSeriesSum}
\end{equation}
even if $N\in\hypNs\subseteq\RC{\sigma}$, and $\RC{\sigma}$ is a
different quotient ring with respect to $\rcrho$ (e.g., in general
we do \emph{not} have $\hypNs\subseteq\hypNr$). In other words, it
is not correct to say that the left hand side of \eqref{eq:pSeriesSum}
is the sum for $1\le n\le N\in\hypNs$ of terms $\frac{1}{n^{p}}\in\RC{\sigma}$,
but instead it is a way to use $(N_{\eps})\in\N_{\sigma}$ to compute
a generalized number of $\rcrho$.

If $p<0$, the general term $\frac{1}{n^{p}}\not\to0$ because $n^{-p}\le\diff{\rho}^{q}$
only if $n<\diff{\rho}^{-q/p}$. Therefore, Lem.~\ref{lem:NtoN+M}
yields that the $p$-series diverges.\, We can hence consider only
the case $p>0$. Let us note that Thm.~\ref{thm:seriesWellDef},
which is clearly based on our Def.~\ref{def:RCringHyperseries} of
$\rcrhos$, allows us to consider partial hypersums of the form $\sum_{n=1}^{N}\frac{1}{n^{p}}$
even if $p<0$. In other words, we can correctly argue that the $p$-hyperseries
does not converge if $p<0$ not because the partial hypersums cannot
be computed, but because the general term does not converge to zero.
The situation would be completely different if we restrict our attention
only to sequences $(a_{n})_{n}$ satisfying the (more natural) uniformly
moderate condition \eqref{eq:uniformlyModerate}. Similar remarks
can be formulated for the calculus of divergent hyperseries in Example
\ref{exa:GeomExpSeries}, \ref{enu:div1}-\ref{enu:div6}.

To prove the following results, we will follow some ideas of \cite{Han05}.
\begin{thm}
\label{thm:ineq-P-Series}Let $p\in\rcrho_{>0}$ and let $S_{N}(p)$
be the $N$-th partial sum of the $p$-hyperseries, where $N\in\hypNs_{\ge1}$,
then
\begin{equation}
1-\frac{1}{2^{p}}+\frac{2}{2^{p}}S_{N}(p)<S_{2N}(p)<1+\frac{2}{2^{p}}S_{N}(p).\label{eq:pGreater0}
\end{equation}
\end{thm}

\begin{proof}
Let $N=[N_{\eps}]\in\hypNs_{\ge1}$, with $N_{\eps}\in\N_{\ge1}$
for all $\eps$, and $p=[p_{\eps}]$, where $p_{\eps}>0$ for all
$\eps$. Since 
\begin{align*}
S_{2N}(p) & =\left[\sum_{n=1}^{2N_{\eps}}\frac{1}{n^{p_{\eps}}}\right]=\left[1+\frac{1}{2^{p_{\eps}}}+\frac{1}{3^{p_{\eps}}}+\dots+\frac{1}{(2N_{\eps}))^{p_{\eps}}}\right]=\\
 & =1+\left[\frac{1}{2^{p_{\eps}}}+\frac{1}{4^{p_{\eps}}}+\dots+\frac{1}{((2N_{\eps}))^{p_{\eps}}}\right]+\left[\frac{1}{3^{p_{\eps}}}+\frac{1}{5^{p_{\eps}}}+\dots+\frac{1}{((2N_{\eps}-1))^{p_{\eps}}}\right].
\end{align*}

\noindent But
\begin{equation}
1+\left[\frac{1}{2^{p_{\eps}}}+\frac{1}{4^{p_{\eps}}}+\dots+\frac{1}{((2N_{\eps}))^{p_{\eps}}}\right]=1+\frac{1}{2^{p}}S_{N}(p)\label{eq:even}
\end{equation}
and, since $p_{\eps}>0$, we also have
\begin{align}
\left[\frac{1}{3^{p_{\eps}}}+\frac{1}{5^{p_{\eps}}}+\dots+\frac{1}{((2N_{\eps}-1))^{p_{\eps}}}\right] & >\left[\frac{1}{4^{p_{\eps}}}+\frac{1}{6^{p_{\eps}}}+\dots+\frac{1}{((2N_{\eps}))^{p_{\eps}}}\right]\nonumber \\
 & =\frac{1}{2^{p}}S_{N}(p)-\frac{1}{2^{p}}.\label{eq:odd}
\end{align}
Summing \eqref{eq:even} and \eqref{eq:odd} yields $S_{2N}(p)>1-\frac{1}{2^{p}}+\frac{2}{2^{p}}S_{N}(p)$.
Finally, from \eqref{eq:even} and
\begin{align*}
\left[\frac{1}{3^{p_{\eps}}}+\frac{1}{5^{p_{\eps}}}+\dots+\frac{1}{((2N_{\eps}-1))^{p_{\eps}}}\right] & <\left[\frac{1}{2^{p_{\eps}}}+\frac{1}{4^{p_{\eps}}}+\dots+\frac{1}{((2N_{\eps}))^{p_{\eps}}}\right]=\\
 & =\frac{1}{2^{p}}S_{N}(p)
\end{align*}
we get $S_{2N}(p)<1+\frac{1}{2^{p}}S_{N}(p)+\frac{1}{2^{p}}S_{N}(p)$.
\end{proof}
\begin{thm}
\label{thm:pSeriesConv}The $p$-hyperseries is divergent when $0<p\sbpt{\le}1$.
When $p>1$, there exist some gauge $\tau\le\rho$ such that $\left(\frac{1}{n^{p}}\right)_{n}\in\rcrho_{\tau}$
and the $p$-series is convergent with respect to the gauges $\tau$,
$\rho$ (to the usual real value $\zeta(p)$ if $p\in\R$), and in
this case

\begin{equation}
\frac{2^{p}-1}{2^{p}-2}\leq\parthypersumarg{\rho}{\tau}{\hspace{1em}n}{\ge1}\frac{1}{n^{p}}=\sum_{n=1}^{+\infty}\frac{1}{n^{p}}\leq\frac{2^{p}}{2^{p}-2}.\label{eq:p-seriesEstimate}
\end{equation}
\end{thm}

\begin{proof}
If $0<p\sbpt{\le}1$, there exist $J\subzero I$ such that $p|_{J}\le1$.\emph{
}Assume that the $p-$hyperseries is convergent and set $\hyperlim{\rho}{\sigma}S_{N}(p)=:S(p)$.
Taking $N\rightarrow\infty$ in the first inequality in \eqref{eq:pGreater0},
we have

\[
1-\text{\ensuremath{\frac{1}{2^{p}}+\frac{2}{2^{p}}S(p)}}=\frac{2^{p}-1}{2^{p}}+\frac{2}{2^{p}}S(p)\leq S(p).
\]
Since for $\eps\in J$ sufficiently small we have $0<p_{\eps}<1$
for some $[p_{\eps}]=p$, we obtain
\[
0<\frac{2^{p_{\eps}}-1}{2^{p_{\eps}}}\leq\frac{2^{p_{\eps}}-2}{2^{p_{\eps}}}S(p)\leq0.
\]
This shows that the $p-$hyperseries is divergent when $p\sbpt{\le}1$.
Now let $p>1$, so that we can assume $p_{\eps}>1$ for all $\eps$.
From \eqref{eq:pGreater0}, we have

\[
S_{N_{\eps}}(p_{\eps})<S_{2N_{\eps}}(p_{\eps})<1+\frac{2}{2^{p_{\eps}}}S_{N_{\eps}}(p_{\eps}),
\]
where $S_{N_{\eps}}(p_{\eps})=\sum_{n=1}^{N_{\eps}}\frac{1}{n^{p_{\eps}}}\in\R$.
Thereby

\[
0<\left(1-\frac{2}{2^{p_{\eps}}}\right)S_{N_{\eps}}(p_{\eps})<1
\]

\[
S_{N_{\eps}}(p_{\eps})<\frac{2^{p_{\eps}}}{2^{p_{\eps}}-2}.
\]
So $S_{N_{\eps}}(p)$ is bounded and increasing, and applying Thm.~\ref{thm:epsConv},
we get the existence of a gauge $\tau$ such that $\left(\frac{1}{n^{p}}\right)_{n}\in\rcrho_{\tau}$
and $\parthypersumarg{\rho}{\tau}{n}{\ge1}\frac{1}{n^{p}}$ converges
to $\sum_{n=1}^{+\infty}\frac{1}{n^{p}}=\zeta(p)$.
\end{proof}
Later, in Cor.~\ref{cor:p-seriesIntTest} and using the integral
test Thm\@.~\ref{thm:intTest}, we will prove the convergence of
$p$-hyperseries for arbitrary gauges $\sigma$, $\rho$ (even if
we will not get the estimates \eqref{eq:p-seriesEstimate}).

\subsection{Absolute convergence}
\begin{thm}
If $\hypersum{\rho}{\sigma}|a_{n}|$ converges then also $\hypersum{\rho}{\sigma}a_{n}$
converges.
\end{thm}

\begin{proof}
Cauchy criterion \cite[Thm.~37]{MTAG}, yields

\[
\forall q\in N\,\exists M\in\hyperN{\sigma}\,\forall N_{1},N_{2}\in\hyperN{\sigma}_{\geq M}:\ \left|\sum_{n=0}^{N_{1}}|a_{n}|-\sum_{n=0}^{N_{2}}|a_{n}|\right|<\diff{\rho^{q}}.
\]

\noindent The conclusion hence follows from the inequality 
\[
\left|\sum_{n=0}^{N_{1}}a_{n}-\sum_{n=0}^{N_{2}}a_{n}\right|<\left|\sum_{n=0}^{N_{1}}|a_{n}|-\sum_{n=0}^{N_{2}}|a_{n}|\right|,
\]

\noindent and from Cauchy criterion for hypersequences, i.e.~\cite[Thm.~37]{MTAG}.
\end{proof}

\subsection{Direct and limit comparison tests}

In the direct comparison test, we need to assume a relation of the
form $a_{n}\le b_{n}$ for all $n\in\N$ between general terms of
two hyperseries. If $a_{n}=[a_{n\eps}]$, $b_{n}=[b_{n\eps}]\in\rcrho$
are two representatives, then this inequality would yield
\[
\exists[z_{n\eps}]=0\in\rcrho\,\exists\eps_{0n}\,\forall\eps\le\eps_{0n}:\ a_{n\eps}\le b_{n\eps}+z_{n\eps}.
\]
The dependence of $\eps_{0n}$ from $n\in\N$ does not allow to prove,
e.g., that $\sum_{n=0}^{N}a_{n}\le\sum_{n=0}^{N}b_{n}$ (e.g.~as
in \eqref{eq:reprCountEx2}, we can consider $b_{n\eps}:=0$ and $a_{n\eps}:=0$
if $\eps<\frac{1}{n+1}$ but $a_{n\eps}:=1$ otherwise, then $a_{n}=[a_{n\eps}]=0=b_{n}$,
but $\sum_{n=0}^{N}a_{n}=N-\left[\left\lceil \frac{1}{\eps}\right\rceil \right]+2$).

Moreover, for convergence tests we also need to perform pointwise
(in $n\in\N$) operations of the form $\left(\frac{a_{n}}{b_{n}}\right)_{n}$,
$\left(\left|\frac{a_{n+k}}{a_{n}}\right|\right)_{n}$ or $\left(\left|a_{n}\right|^{1/n}\right)_{n}$.
This kind of pointwise operations can be easily considered if we restrict
us to sequences $\{a_{n}\}_{n}=[a_{n\eps}]_{\text{u}}\in\rcrhou$
and we hence assume that $\sigma\ge\rho^{*}$. Anyway, this is a particular
sufficient condition, and we can more generally state some convergence
tests using the more general space $\rcrhos$. For these reasons,
we define
\begin{defn}
\label{def:orderR_s}Let $(a_{n})_{n}$, $(b_{n})_{n}\in\rcrhos$,
then we say that $(a_{n})_{n}\le(b_{n})_{n}$ if
\begin{equation}
\forall N,M\in\hypNs:\ \sum_{n=N}^{M}a_{n}\le\sum_{n=N}^{M}b_{n}\text{ in }\rcrho.\label{eq:inequality}
\end{equation}
Note explicitly that if $M<_{L}N$ on $L\subzero I$, then $\sum_{n=N}^{M}a_{n}=_{L}\sum_{n=N}^{M}b_{n}=_{L}0$.
Therefore, \eqref{eq:inequality} is equivalent to
\[
\forall N,M\in\hypNs:\ M\ge N\ \Rightarrow\ \sum_{n=N}^{M}a_{n}\le\sum_{n=N}^{M}b_{n}\text{ in }\rcrho.
\]
\end{defn}

On the other hand, we recall that the Robinson-Colombeau ring $\rcrhou$
is already ordered by $\left\{ a_{n}\right\} _{n}=[a_{n\eps}]_{\text{u}}\le[b_{n\eps}]_{\text{u}}=\left\{ b_{n}\right\} _{n}$
if
\begin{equation}
\exists\left[z_{n\eps}\right]_{\text{u}}=0\in\rcrhou\,\forall^{0}\eps\,\forall n\in\N:\ a_{n\eps}\le b_{n\eps}+z_{n\eps}.\label{eq:inequalityRcrhou}
\end{equation}
We also recall that $\rcrho\subseteq\rcrhos$ through the embedding
$x\mapsto[(x_{\eps})_{n\eps}]$. Thereby, a relation of the form $x\le[a_{n\eps}]_{\text{u}}$
yields
\[
\forall^{0}\eps\,\forall n\in\N:\ x_{\eps}\le a_{n\eps}
\]
for some representative $x=[x_{\eps}]$. Finally, \cite[Lem.~2]{MTAG}
for the ring $\rcrhou$ implies that $\{a_{n}\}_{n}<\{b_{n}\}_{n}$
if and only if
\[
\exists q\in\R_{>0}\,\forall^{0}\eps\,\forall n\in\N:\ b_{n\eps}-a_{n\eps}\ge\diff{\rho}^{q}.
\]

\begin{thm}
\label{thm:orderedModule}We have the following properties:
\begin{enumerate}
\item \label{enu:orderedMod}$(\rcrhos,\le)$ is an ordered $\rcrho$-module.
\item \label{enu:positiveIncrease}If $(a_{n})_{n}\ge0$, then $N\in\hypNs\mapsto\sum_{n=0}^{N}a_{n}\in\rcrho$
is increasing.
\item \label{enu:orderPres}If $\sigma\ge\rho^{*}$, and $[a_{n\eps}]_{\text{\emph{u}}}\le[b_{n\eps}]_{\text{\emph{u}}}$
in $\rcrhou$, then $[a_{n\eps}]_{\text{\emph{s}}}\le[b_{n\eps}]_{\text{\emph{s}}}$
in $\rcrhos$.
\end{enumerate}
\end{thm}

\begin{proof}
\ref{enu:orderedMod}: The relation $\le$ on $\rcrhos$ is clearly
reflexive and transitive. If $(a_{n})_{n}\le(b_{n})_{n}\le(a_{n})_{n}$,
then for all $N$, $M\in\hypNs$ we have $\sum_{n=N}^{M}a_{n}=\sum_{n=N}^{M}b_{n}$
in $\rcrho$. From Def.~\ref{def:hyperseries} of hypersum, this
implies \eqref{eq:negligibleHyperseries}, i.e.~that $(a_{n})_{n}=(b_{n})_{n}$
in $\rcrhos$.

\ref{enu:positiveIncrease}: Let $N$, $M\in\hypNs$ with $N\le M$,
then $\sum_{n=0}^{M}a_{n}=\sum_{n=0}^{N}a_{n}+\sum_{n=N+1}^{M}a_{n}$
by Lem.~\ref{lem:N+M}, and $\sum_{n=N+1}^{M}a_{n}\ge0$ because
$(a_{n})_{n}\ge0$ in $\rcrhos$.

\ref{enu:orderPres}: Assume that the inequality in \eqref{eq:inequalityRcrhou}
holds for all $\eps\le\eps_{0}$ and for all $n\in\N$. Then, if $N=[N_{\eps}]$,
$M=[M_{\eps}]$, $N_{\eps}$, $M_{\eps}\in\N$, we have $\sum_{n=N_{\eps}}^{M_{\eps}}a_{n\eps}\le\sum_{n=N_{\eps}}^{M\eps}b_{n\eps}+\sum_{n=N_{\eps}}^{M\eps}z_{n\eps}$
for all $\eps\le\eps_{0}$. Since $\left[z_{n\eps}\right]_{\text{u}}=0\in\rcrhou$,
for each $q\in\N$, for $\eps$ small and for all $n\in\N$, we have
$|z_{n\eps}|\le\rho_{\eps}^{q}$ and hence $\left|\sum_{n=N_{\eps}}^{M_{\eps}}z_{n\eps}\right|\le\left|M_{\eps}-N_{\eps}\right|\rho_{\eps}^{q}$.
Thereby, from the assumption $\sigma\ge\rho^{*}$, it follows that
$\left(\sum_{n=N_{\eps}}^{M_{\eps}}z_{n\eps}\right)\sim_{\rho}0$,
which proves the conclusion.
\end{proof}
The direct comparison test for hyperseries can now be stated as follows.
\begin{thm}
\label{thm:directComp}Let $\sigma$ and $\rho$ be arbitrary gauges,
let $\hypersum{\rho}{\sigma}a_{n}$ and $\hypersum{\rho}{\sigma}b_{n}$
be hyperseries with $\left(a_{n}\right)_{n}$, $\left(b_{n}\right)_{n}\ge0$
and such that
\begin{equation}
\exists N\in\N:\ (a_{n+N})_{n}\le(b_{n+N})_{n}.\label{eq:rel_a_b}
\end{equation}
Then we have:
\begin{enumerate}
\item \label{enu:directCompConv}If $\hypersum{\rho}{\sigma}b_{n}$ is convergent
then so is $\hypersum{\rho}{\sigma}a_{n}$.
\item \label{enu:directCompDiv}If $\hypersum{\rho}{\sigma}a_{n}$ is divergent
to $+\infty$, then so is $\hypersum{\rho}{\sigma}b_{n}$.
\item \label{enu:directCompRcrhou}If $\sigma\ge\rho^{*}$ and $\{a_{n}\}_{n}:=[a_{n\eps}]_{\text{\emph{u}}}$,
$\{b_{n}\}_{n}=[b_{n\eps}]_{\text{\emph{u}}}\in\rcrhou$, then we
have the same conclusions if we assume that
\[
\exists N\in\N:\ \{a_{n+N}\}_{n}\le\{b_{n+N}\}_{n}\text{ in }\rcrhou.
\]
\end{enumerate}
\end{thm}

\begin{proof}
We first note that \eqref{eq:rel_a_b} can be simply written as $(\alpha_{n})_{n}\le(\beta_{n})_{n}$,
where $\alpha_{n}:=a_{n+N}$ and $\beta_{n}:=b_{n+N}$. Thereby, Cor.~\ref{cor:convAddRem}
implies that, without loss of generality, we can assume $N=0$. To
prove \ref{enu:directCompConv}, let us consider the partial sums
$A_{N}:=\sum_{i=0}^{N}a_{i}$ and $B_{N}:=\sum_{i=0}^{N}b_{i}$, $N\in\hypNs$.
Since $\hypersum{\rho}{\sigma}b_{n}$ is convergent, we have $\exists\hyperlimarg{\rho}{\sigma}{N}B_{N}=:B\in\rcrho$.
The assumption $\left(a_{n}\right)_{n}\le\left(b_{n}\right)_{n}$
implies $A_{N}\le B_{N}$. The hypersequences $\left(A_{N}\right)_{N\in\hypNs}$,
$\left(B_{N}\right)_{N\in\hypNs}$ are increasing because $(a_{n})_{n}$,
$(b_{n})_{n}\ge0$ (Thm,~\ref{thm:orderedModule}.\ref{enu:positiveIncrease})
and hence $\left(B-B_{N}\right)_{N\in\hypNs}$ decreases to zero because
of the convergence assumption. Now, for all $N$, $M\in\hypNs$, $M\ge N$,
we have
\begin{equation}
A_{N}\leq A_{M}=\sum_{n=0}^{M}a_{n}=\sum_{n=0}^{N}a_{n}+\sum_{n=N+1}^{M}a_{n}\le A_{N}+\sum_{n=N+1}^{M}b_{n}=A_{N}+(B-B_{N}).\label{eq:interval}
\end{equation}
Thereby, given $m$, $n\ge N$, applying \eqref{eq:interval} with
$m$, $n$ instead of $M$, we get that both $A_{n}$, $A_{m}$ belong
to the interval $[A_{N},A_{N}+(B-B_{N})]$, whose length $B-B_{N}$
decreases to zero as $N\in\hypNs$ goes to infinity. This shows that
$(A_{n})_{n\in\hypNs}$ is a Cauchy hypersequence, and therefore $\hypersum{\rho}{\sigma}a_{n}$
converges.

The proof of \ref{enu:directCompDiv} follows directly from the inequality
$A_{N}\le B_{N}$ for each $N\in\hypNs$. The proof of \ref{enu:directCompRcrhou}
follows from Lem.~\ref{thm:orderedModule}.\ref{enu:orderPres} and
from Lem.~\ref{lem:Rtilu-RtilsComparison}.\ref{enu:convRtilu}.
\end{proof}
\noindent Note that the two cases \ref{enu:directCompConv} and \ref{enu:directCompDiv}
do not cover the case where the non-decreasing hypersums $N\mapsto\sum_{n=0}^{N}a_{n}$
are bounded but anyway do not converge because it does not exists
the supremum of their values (see \cite{MTAG}).
\begin{example}
\label{exa:direct}Let $x\in\rcrho_{\ge0}$ be a finite number so
that $x\le M\in\N_{>0}$ for some $M$. Assume that $x_{\eps}\le M$
for all $\eps\le\eps_{0}$. For these $\eps$ we have $\frac{n+x_{\eps}}{n^{3}}\le\frac{2}{n^{2}}$
for all $n\in\N_{\ge M}$. This shows that $\left\{ \frac{n+M+x}{(n+M)^{3}}\right\} _{n}\le\left\{ \frac{2}{(n+M)^{2}}\right\} _{n}$
and we can hence apply Thm.~\ref{thm:directComp}.
\end{example}

The limit comparison test is the next
\begin{thm}
Let $\sigma$ and $\rho$ be arbitrary gauges, let $\hypersum{\rho}{\sigma}a_{n}$
and $\hypersum{\rho}{\sigma}b_{n}$ be hyperseries, with $\left(a_{n}\right)_{n}\geq0$.
Assume that
\begin{equation}
\exists m,M\in\rcrho_{>0}\,\exists N\in\N:\ \left(mb_{n+N}\right)_{n}\le\left(a_{n+N}\right)_{n}\le\left(Mb_{n+N}\right)_{n}.\label{eq:limitTest}
\end{equation}
Then we have:
\begin{enumerate}
\item \label{enu:limComp}Either both hyperseries $\hypersum{\rho}{\sigma}a_{n}$,
$\hypersum{\rho}{\sigma}b_{n}$ converge or diverge to $+\infty$.
\item \label{enu:limCompRcrhou}If $\sigma\ge\rho^{*}$, $\left\{ a_{n}\right\} _{n}=[a_{n\eps}]_{\text{\emph{u}}}$,
$\left\{ b_{n}\right\} _{n}=[b_{n\eps}]_{\text{\emph{u}}}\in\rcrhou$,
$[b_{n\eps}]_{\text{\emph{u}}}>0$ and
\[
\exists m,M\in\rcrho_{>0}\,\exists N\in\N:\ m\le\left\{ \frac{a_{n+N}}{b_{n+N}}\right\} _{n}\le M,
\]
then the same conclusion as in \ref{enu:limComp} holds.
\end{enumerate}
\begin{proof}
As in the previous proof, without loss of generality, we can assume
$N=0$. Now, if $\hypersum{\rho}{\sigma}b_{n}$ diverges to $+\infty$,
then so does $\hypersum{\rho}{\sigma}mb_{n}$ because $m>0$. Since
$\left(mb_{n}\right)_{n}<\left(a_{n}\right)_{n}$, by the direct comparison
test also the hyperseries $\hypersum{\rho}{\sigma}a_{n}$ diverges
to $+\infty$. Likewise, if the hyperseries $\hypersum{\rho}{\sigma}b_{n}$
converges then so does $\hypersum{\rho}{\sigma}Mb_{n}$. Since $\left(a_{n}\right)_{n}<\left(Mb_{n}\right)_{n}$,
by the direct comparison test also the hyperseries $\hypersum{\rho}{\sigma}a_{n}$
converges. Property \ref{enu:limCompRcrhou} can be proved as in the
previous theorem.
\end{proof}
\end{thm}

\noindent Note that if $M$, $a_{n}$, $b_{n}\in\R$, then the condition
$\left\{ \frac{a_{n}}{b_{n}}\right\} _{n}\le M$ is equivalent to
$\limsup_{n\to+\infty}\frac{a_{n}}{b_{n}}\le M$. Analogously, $m\le\left\{ \frac{a_{n}}{b_{n}}\right\} _{n}$
is equivalent to $\liminf_{n\to+\infty}\frac{a_{n}}{b_{n}}\ge m$.
This shows that our formulation of the limit comparison test faithfully
generalizes the classical version.

Both the root and the ratio tests are better formulated in $\rcrhou$,
because their proofs require term-by-term operations (such as e.g.~that
$|a_{n}|^{1/n}\le L$ implies $|a_{n}|\le L^{n}$), so that the order
relation defined in Def.~\eqref{def:orderR_s} seems too weak for
these aims; see also the comment below the proof of the root test.

\subsection{Root test}
\begin{thm}
Let $\rho^{*}\le\sigma\le\rho^{*}$ and let $\left\{ |a_{n}|\right\} _{n}\in\rcrhou$
(so that also $\left\{ \left|a_{n}\right|^{1/n}\right\} _{n}\in\rcrhou$).
Assume that
\begin{equation}
\exists L\in\rcrho:\ \left\{ \left|a_{n}\right|^{1/n}\right\} _{n}\le L.\label{eq:root}
\end{equation}
Then
\begin{enumerate}
\item \label{enu:rootLess1}If $L<1$, then the hyperseries $\hypersum{\rho}{\sigma}a_{n}$
converges absolutely.
\item \label{enu:rootGr1Subpt}If $J\subzero I$ and $L>_{J}1$, then $\left(\left|a_{n}\right||_{J}\right)_{n\in\hypNs}\to+\infty$
and hence the hyperseries $\hypersum{\rho}{\sigma}a_{n}|_{J}$ diverges.
\item \label{enu:rootIff}If $L\not\sbpt{=}1$, then the hyperseries $\hypersum{\rho}{\sigma}a_{n}$
converges if and only if $L<1$.
\end{enumerate}
\end{thm}

\begin{proof}
\ref{enu:rootLess1}: Our assumption $\left\{ |a_{n}|^{1/n}\right\} _{n}\le L$
entails $\left\{ |a_{n}|\right\} _{n}<\left(L^{n}\right)_{n}$. Since
$\parthypersumarg{\rho}{\sigma}{n}{\ge0}L^{n}$ is convergent, because
$0\le L<1$ and $\sigma\le\rho^{*}$ (see Example \ref{exa:GeomExpSeries}.\ref{enu:geomSeries}),
by the direct comparison test, the series $\hypersum{\rho}{\sigma}|a_{n}|$
is also convergent.

\ref{enu:rootGr1Subpt}: Now, assume that $L_{\eps}>1$ for $\eps\in J\subzero I$
and let us work directly in the ring $\rcrho|_{J}$. Proceeding as
above, we have $\left\{ |a_{n}|\right\} _{n}\ge_{J}\left(L^{n}\right)_{n}$,
and thereby the last claim follows because of $L>_{J}1$.

\ref{enu:rootIff}: Assume that $L\not\sbpt{=}1$, that $\hypersum{\rho}{\sigma}a_{n}$
converges but $L\sbpt{>}1$, then \ref{enu:rootGr1Subpt} would yield
that $\hypersum{\rho}{\sigma}a_{n}|_{J}$ diverges for some $J\subzero I$
and this contradicts the convergence assumption. Therefore, we have
$L\not\sbpt{=}1$ and $L\not\sbpt{>}1$, and hence \cite[Lem.~6.(v)]{MTAG}
gives $L\not\sbpt{\ge}1$. Thereby, \cite[Lem.~5.(ii)]{MTAG} finally
implies $L<1$.
\end{proof}
\noindent Actually, we can also simply reformulate the root test in
$\rcrhos$ by trivially asking that $\left(\left|a_{n}\right|\right)_{n}\le\left(L^{n}\right)_{n}$,
but that would simply recall a non-meaningful consequence of the direct
comparison test Thm.~\ref{thm:directComp}.

\subsection{Ratio test}

Proceeding as in the previous proof, i.e.~by generalizing the classical
proof for series of real numbers, we also have the following
\begin{thm}
Let $\rho^{*}\le\sigma\le\rho^{*}$ and let $\left\{ a_{n}\right\} _{n}\in\rcrhou$,
with $\left\{ |a_{n}|\right\} _{n}>0$. Assume that for some $k\in\N$
we have
\[
\exists L\in\rcrho:\ \left\{ \frac{a_{n+k}}{a_{n}}\right\} _{n}\le L.
\]
Then
\begin{enumerate}
\item If $L<1$, the hyperseries $\hypersum{\rho}{\sigma}a_{n}$ converges.
\item If $J\subzero I$, $L>_{J}1$, then $\left(\left|a_{n}\right||_{J}\right)_{n\in\hypNs}\to+\infty$
and hence the hyperseries $\hypersum{\rho}{\sigma}a_{n}|_{J}$ diverges.
\item If $L\not\sbpt{=}1$, then the hyperseries $\hypersum{\rho}{\sigma}a_{n}$
converges if and only if $L<1$.
\end{enumerate}
\end{thm}

As in the classical case, the convergent $p$-hyperseries $\hypersum{\rho}{\sigma}\frac{1}{n^{2}}$
and the divergent hyperseries $\hypersum{\rho}{\sigma}1^{n}$ shows
that the ratio and the root tests fails if $L=1$.
\begin{example}
If $x\in\rcrho$ is a finite invertible number, then using the ratio
test and proceeding as in Example \ref{exa:direct}, we can prove
that $\hypersum{\rho}{\sigma}\frac{x^{n}}{n!}$ converges if $\sigma\ge\rho^{*}$.
\end{example}

\subsection{Alternating series test}

Also in the proof of this test we need a term-by-term comparison (see
\eqref{eq:monRepProof} below).
\begin{thm}
\label{thm:altTest}Let $\sigma\ge\rho^{*}$ and let $\left\{ a_{n}\right\} _{n}\in\rcrhou$.
Assume that we have
\begin{align}
 & \left\{ a_{n+1}\right\} _{n}\le\left\{ a_{n}\right\} _{n}\label{eq:monotRepr}\\
 & \hyperlim{\rho}{\sigma}a_{n}=0.\label{eq:zeroLim}
\end{align}
Then the alternating hyperseries $\hypersum{\rho}{\sigma}(-1)^{n}a_{n}$
converges.
\end{thm}

\begin{proof}
By \eqref{eq:monotRepr}, assume that
\begin{equation}
a_{n+1,\eps}\le a_{n\eps}\quad\forall n\in\N\,\forall\eps\le\eps_{0}\label{eq:monRepProof}
\end{equation}
 holds for all $n\in\N$ and all $\eps\le\eps_{0}$. Take $N=[N_{\eps}]$,
$M=[M_{\eps}]\in\hypNs$, where for $\eps\le\eps_{0}$ we have $N_{\eps}$,
$M_{\eps}\in\N$. If $M_{\eps}$ is odd and $M_{\eps}\le N_{\eps}$,
we can estimate the difference $S_{N_{\eps}}-S_{M_{\eps}}$ as:
\begin{equation}
\begin{aligned}S_{N_{\eps}}-S_{M_{\eps}} & =\sum_{n=0}^{N_{\eps}}(-1)^{n}a_{n}-\sum_{n=0}^{M_{\eps}}(-1)^{n}a_{k}=\sum_{n=M_{\eps}+1}^{N_{\eps}}(-1)^{n}a_{n}=\\
 & =a_{M_{\eps}+1}-a_{M_{\eps}+2}+a_{M_{\eps}+3}+\cdots+a_{N_{\eps}}=\\
 & ={\displaystyle a_{M_{\eps}+1}-(a_{M_{\eps}+2}-a_{M_{\eps}+3})-(a_{M_{\eps}+4}-a_{M_{\eps}+5})-\cdots-a_{N_{\eps}}\leq}\\
 & \le a_{M_{\eps}+1}\leq a_{M_{\eps}},
\end{aligned}
\label{eq:CauchyAlt}
\end{equation}
where we used \eqref{eq:monRepProof}. If $M_{\eps}$ is even and
$M_{\eps}\le N_{\eps}$, a similar argument shows that $S_{N_{\eps}}-S_{M_{\eps}}\ge-a_{M_{\eps}}$.
If $M_{\eps}>N_{\eps}$, it suffices to revert the role of $M_{\eps}$
and $N_{\eps}$ in \eqref{eq:CauchyAlt}. This proves that $\min(-a_{M},-a_{N})\le S_{N}-S_{M}\leq\max(a_{M},a_{N})$.
The final claim now follows from \eqref{eq:zeroLim} and Cauchy criterion
\cite[Thm.~44]{MTAG}.
\end{proof}
\begin{example}
Using the previous Thm.~\ref{thm:altTest}, we can prove, e.g., that
the hyperseries $\hypersum{\rho}{\sigma}\frac{(-1)^{n}}{n}$ is convergent.
\end{example}

\subsection{Integral test}

The possibility to prove the integral test for hyperseries is constrained
by the existence of a notion of generalized function that can be defined
on an unbounded interval, e.g.~of the form $[0,+\infty)\subseteq\rcrho$.
This is not possible for an arbitrary Colombeau generalized function,
which are defined only on finite (i.e.~compactly supported) points,
i.e.~on domains of the form $\widetilde{\Omega}_{c}$ (see e.g.~\cite{GKOS}).
Moreover, we would also need a notion of improper integral extended
on $[0,+\infty)$. This notion clearly needs to have good relations
with the notion of hyperlimit and with the definite integral over
the interval $[0,N]$, where $N\in\hypNr$ (see Def.~\ref{def:imprIntegral}
below). This further underscores drawbacks of Colombeau's theory,
where this notion of integral (if $N$ is an infinite number) is not
defined, see e.g.~\cite{AFJO}.

The theory of generalized smooth functions overcomes these difficulties,
see e.g.~\cite{GKV,TI}. Here, we only recall the equivalent definition.
In the following, $\sigma$, $\rho$ always denote arbitrary gauges.
\begin{defn}
Let $X\subseteq\RC{\sigma}^{n}$ and $Y\subseteq\RC{\sigma}^{d}$.
We say that $f:X\longrightarrow Y$ is a $\rho$-moderate \emph{generalized
smooth function}\textcolor{red}{{} }(GSF), and we write $f\in\gsfud{\rho}{\sigma}(X,Y)$,
if
\begin{enumerate}
\item $f:X\ra Y$ is a set-theoretical function.
\item There exists a net $(f_{\eps})\in\Coo(\R^{n},\R^{d})^{(0,1]}$ such
that for all $[x_{\eps}]\in X$:
\begin{enumerate}[label=(\alph*)]
\item $f(x)=[f_{\eps}(x_{\eps})]$
\item $\forall\alpha\in\N^{n}:\ (\partial^{\alpha}f_{\eps}(x_{\eps}))\text{ is }\rho-\text{moderate}$.
\end{enumerate}
\end{enumerate}
For generalized smooth functions lots of results hold: closure with
respect to composition, embedding of Schwartz's distributions, differential
calculus, one-di\-men\-sio\-nal integral calculus using primitives,
classical theorems (intermediate value, mean value, Taylor, extreme
value, inverse and implicit function), multidimensional integration,
Banach fixed point theorem, a Picard-Lindelöf theorem for both ODE
and PDE, several results of calculus of variations, etc.

In particular, we have the following
\end{defn}

\begin{thm}
\label{thm:existenceUniquenessPrimitives}Let $a$, $b$, $c\in\rcrho$,
with $a<b$ and $c\in[a,b]\subseteq U$. Let $f\in\gsfud{\rho}{\sigma}([a,b],\rcrho)$
be a generalized smooth function. Then, there exists one and only
one generalized smooth function $F\in\gsfud{\rho}{\sigma}([a,b],\rcrho)$
such that $F(c)=0$ and $F'(x)=f(x)$ for all $x\in[a,b]$. Moreover,
if $f$ is defined by the net $f_{\eps}\in\Coo(\R,\R)$ and $c=[c_{\eps}]$,
then
\begin{equation}
F(x)=\left[\int_{c_{\eps}}^{x_{\eps}}f_{\eps}(s)\,\diff{s}\right]\label{eq:intEps}
\end{equation}
for all $x=[x_{\eps}]\in[a,b]$.
\end{thm}

\noindent We can hence define $\int_{c}^{x}f(t)\,\diff{t}:=F(x)$
for all $x\in[a,b]$ (note explicitly that $a$, $b$ can also be
infinite generalized numbers). For this notion of integral we have
all the usual elementary property, monotonicity and integration by
substitution included.

For the integral test for hyperseries, we finally need the following
\begin{defn}
\label{def:imprIntegral}Let $f\in\gsfud{\rho}{\sigma}([0,+\infty),\rcrho)$,
then we say
\begin{align*}
\exists\int_{0}^{+\infty}f(x)\,\diff{x} & \DIff\exists\,\hyperlimarg{\rho}{\sigma}{N}\int_{0}^{N}f(x)\,\diff{x}=:\int_{0}^{+\infty}f(x)\,\diff{x}\in\rcrho\\
\int_{0}^{+\infty}f(x)\,\diff{x}=\pm\infty & \DIff\hyperlimarg{\rho}{\sigma}{N}\int_{0}^{N}f(x)\,\diff{x}=\pm\infty.
\end{align*}
\end{defn}

\begin{thm}
\label{thm:intTest}Let $f=[f_{\eps}(-)]\in\gsfud{\rho}{\sigma}([0,+\infty),\rcrho)$
be a GSF such that
\begin{align}
\forall^{0}\eps\,\forall n & \in\N\,\forall x\in[n,n+1)_{\R}:\ f_{\eps}(x)\le f_{\eps}(n)\label{eq:decr1}\\
\forall^{0}\eps\,\forall n & \in\N_{>0}\,\forall x\in[n-1,n)_{\R}:\ f_{\eps}(n)\le f_{\eps}(x)\label{eq:decr2}
\end{align}
then $\left(f(n)\right)_{n}=\left[f_{\eps}(n)\right]_{\text{\emph{s}}}\in\rcrhos$
and
\begin{enumerate}
\item The hyperseries $\hypersum{\rho}{\sigma}f(n)$ converges if $\exists\int_{0}^{+\infty}f(x)\,\diff{x}$
and $\hyperlimarg{\rho}{\sigma}{n}f(n)=0$.
\item The hyperseries $\hypersum{\rho}{\sigma}f(n)$ diverges to $+\infty$
if $\int_{0}^{+\infty}f(x)\,\diff{x}=+\infty$.
\end{enumerate}
\end{thm}

\begin{proof}
From \eqref{eq:intEps}, for all $N$, $M\in\hypNr$ we have
\[
\int_{N}^{M+1}f(x)\,\diff{x}=\left[\int_{N_{\eps}}^{M_{\eps}+1}f_{\eps}(x)dx\right].
\]
Without loss of generality, we can assume $N_{\eps}$, $M_{\eps}\in\N$
for all $\eps$. Thereby
\begin{align*}
\int_{N}^{M+1}f(x)dx & =\left[\sum_{n=N_{\eps}}^{M_{\eps}}\int_{n}^{n+1}f_{\eps}(x)\,\diff{x}\right]\le\\
 & =\left[\sum_{n=N_{\eps}}^{M_{\eps}}\int_{n}^{n+1}f_{\eps}(n)\,\diff{x}\right]=\sum_{n=N}^{M}f(n),
\end{align*}
where we used Def.~\ref{def:hyperseries} of hypersum and \eqref{eq:decr1}.
On the other hand, since $N\in\hypNs\subseteq[0,+\infty)\subseteq\RC{\sigma}$,
using \eqref{eq:decr2} we have

\begin{align*}
\sum_{n=N}^{M}f(n) & =f(N)+\sum_{n=N+1}^{M}\int_{n-1}^{n}f(n)\,\diff{x}=f(N)+\left[\sum_{n=N_{\eps}+1}^{M_{\eps}}\int_{n-1}^{n}f_{\eps}(n)\,\diff{x}\right]\le\\
 & \le f(N)+\left[\sum_{n=N_{\eps}+1}^{M_{\eps}}\int_{n-1}^{n}f_{\eps}(x)\,\diff{x}\right]=f(N)+\int_{N}^{M}f(x)\,\diff{x},
\end{align*}
where we used again Def.~\ref{def:hyperseries} of hypersum and \eqref{eq:intEps}.
Combining these two inequalities yields

\begin{equation}
\int_{N}^{M+1}f(x)\,\diff{x}\leq\sum_{n=N}^{M}f(n)\leq f(N)+\int_{N}^{M}f(x)\,\diff{x}.\label{eq:intTestCauchy}
\end{equation}
These inequalities show that $(f(n))_{n}\in\rcrhos$. Now, we consider
that $\sum_{n=0}^{M}f(n)-\sum_{n=0}^{N}f(n)=\sum_{n=N}^{M}f(n)$ for
$M\ge N$, and
\[
\int_{N}^{M}f(x)\,\diff{x}=-\int_{0}^{N}f(x)\,\diff{x}+\int_{0}^{M}f(x)\,\diff{x}.
\]
Therefore, if $\exists\int_{0}^{+\infty}f(x)\,\diff{x}$ and $\hyperlimarg{\rho}{\sigma}{n}f(n)=0$,
letting $N$, $M\rightarrow+\infty$ in \eqref{eq:intTestCauchy}
proves that our hyperseries satisfies the Cauchy criterion. If $\int_{0}^{+\infty}f(x)\,\diff{x}=+\infty$,
then setting $N=0$ in the first inequality of \eqref{eq:intTestCauchy}
and $M\to+\infty$ we get the second conclusion.
\end{proof}
Note that, generalizing by contradiction the classical proof, we only
have
\[
\exists\int_{0}^{+\infty}f(x)\,\diff{x}\text{ and }\exists\hyperlimarg{\rho}{\sigma}{n}\left|f(n)\right|=:L\ \ \Rightarrow\ \ L\text{ is not invertible}
\]
and not $L=0$ like in classical case.

As usual, from the integral test we can also deduce the $p$-series
test:
\begin{cor}
\label{cor:p-seriesIntTest}If $p\in\rcrho_{>1}$, then the $p$-hyperseries
converges.
\end{cor}

\begin{proof}
Since $p=[p_{\eps}]>1$, we can assume that $p_{\eps}>1$ for all
$\eps\le\eps_{0}$. For these $\eps$, if $x\in[n,n+1)_{\R}$, we
also have $\frac{1}{x^{p_{\eps}}}\le\frac{1}{n^{p_{\eps}}}$, which
proves \eqref{eq:decr1}. Similarly, we can prove \eqref{eq:decr2}.
Note that the function $f(x)=\frac{1}{x^{p}}$ for all $x\in[1,+\infty)\subseteq\RC{\sigma}$
is $\rho$-moderate because $x\ge1$ and $p>1$, so that $\left|f(x)\right|\le1$.
Finally, like in the classical case, we have $\int_{1}^{+\infty}\frac{\diff{x}}{x^{p}}=\frac{1}{1-p}\hyperlimarg{\rho}{\rho}{N}\left(\frac{1}{N^{p-1}}-1\right)=\frac{1}{p-1}$
because $p>1$.
\end{proof}
\noindent Note, however, that our proofs in Sec.~\ref{subsec:pSeries}
are independent from the notion of generalized smooth function.

\section{\label{sec:Cauchy-product}Cauchy product of hyperseries}

We recall that, on the basis of Thm.~\ref{thm:module}, the operations
of sum and product by a number in $\rcrho$ are the sole operations
defined in $\rcrhos$. We now consider the classical Cauchy product:
\begin{defn}
The \emph{Cauchy product of two sequences for hypersums} $\left(a_{n}\right)_{n}$,
$\left(b_{n}\right)_{n}\in\rcrhos$ is defined as
\[
\left(a_{n}\right)_{n}\star\left(b_{n}\right)_{n}:=\left(\sum_{k=0}^{n}a_{k}b_{n-k}\right)_{n}.
\]
The \emph{Cauchy product of two hyperseries} is the hyperlimit of
the hypersums with general term $\left(a_{n}\right)_{n}\star\left(b_{n}\right)_{n}$,
assuming that it exists:
\[
\left(\hypersum{\rho}{\sigma}a_{n}\right)*\left(\hypersum{\rho}{\sigma}b_{n}\right):=\hypersum{\rho}{\sigma}\left(a_{n}\right)_{n}\star\left(b_{n}\right)_{n}.
\]
A similar \emph{product of sequences} can be defined in $\rcrhou$
with
\[
\left\{ a_{n}\right\} _{n}\star\left\{ b_{n}\right\} _{n}:=\left\{ \sum_{k=0}^{n}a_{k}b_{n-k}\right\} _{n}.
\]
\end{defn}

\begin{thm}
\label{thm:ringCauchy}The module $\rcrhos$ is a ring with respect
to the Cauchy product of sequences. The ring $\rcrhou$ is also a
ring with respect to the Cauchy product of sequences.
\end{thm}

\begin{proof}
Let $\left(a_{n}\right)_{n}$, $\left(b_{n}\right)_{n}\in\rcrhou$,
$(a_{n\eps})_{n,\eps}\sim_{\sigma\rho}(\bar{a}_{n\eps})_{n,\eps}$,
$(b_{n\eps})_{n,\eps}\sim_{\sigma\rho}(\bar{b}_{n\eps})_{n,\eps}$.
We first prove that the Cauchy product $\left(a_{n}\right)_{n}\star\left(b_{n}\right)_{n}$
is well-defined in $\rcrhos$. For simplicity, for $n$, $m\in\hypNs$,
set $A_{nm}:=\sum_{k=n}^{m}a_{k}$, $B_{nm}:=\sum_{k=n}^{m}b_{k}$
and similar notations $\bar{A}_{nm}$, $\bar{B}_{nm}$ using the other
representatives. Then
\[
\sum_{k=n}^{m}a_{k}b_{n-k}-\sum_{k=n}^{m}\bar{a}_{k}\bar{b}_{n-k}=(A_{nm}-\bar{A}_{nm})B_{nm}+\bar{A}_{nm}(B_{nm}-\bar{B}_{nm}).
\]
Now, $A_{nm}-\bar{A}_{nm}$ and $B_{nm}-\bar{B}_{nm}$ are $\rho$-negligible
by definition \eqref{eq:negligibleHyperseries} of $\sim_{\sigma\rho}$,
and $B_{nm}$, $\bar{A}_{nm}$ are $\rho$-moderate because of Lem.~\ref{thm:seriesWellDef}.
This proves that the Cauchy product $\star$ in $\rcrhos$ is well-defined.
Similarly, we can proceed with $\rcrhou$. The ring properties now
follow from the same properties of the ring $\rcrho$.
\end{proof}
\noindent It is well-known that for each $n\in\N$, we have $\left(a_{n}\right)_{n}\star\left(b_{n}\right)_{n}=\left(\sum_{i=0}^{n}a_{i}\right)\cdot\left(\sum_{j=0}^{n}b_{j}\right)=\sum_{i=0}^{n}\sum_{j=0}^{n}a_{i}b_{j}=:c_{n}$,
so $(c_{n})_{n}\in\rcrhos$ by Thm.~\ref{thm:ringCauchy}. Moreover,
if $\hypersum{\rho}{\sigma}a_{n}$ and $\hypersum{\rho}{\sigma}b_{n}$
converge, taking the hyperlimit of the product of the two partial
sums we obtain
\begin{align*}
\left(\hypersum{\rho}{\sigma}a_{n}\right)\cdot\left(\hypersum{\rho}{\sigma}b_{n}\right) & =\hyperlimarg{\rho}{\sigma}{N}c_{N}=\\
 & =\hyperlimarg{\rho}{\sigma}{N}\sum_{i=0}^{N}\sum_{j=0}^{N}a_{i}b_{j}
\end{align*}
but, in general, this is different from
\begin{align*}
\hypersum{\rho}{\sigma}c_{n} & =\hyperlimarg{\rho}{\sigma}{N}\sum_{n=0}^{N}c_{n}=\hyperlimarg{\rho}{\sigma}{N}\sum_{n=0}^{N}\sum_{i=0}^{n}\sum_{j=0}^{i}a_{i}b_{j}=\\
 & =\hyperlimarg{\rho}{\sigma}{N}\sum_{n=0}^{N}\sum_{k=0}^{n}a_{k}b_{n-k}=\left(\hypersum{\rho}{\sigma}a_{n}\right)*\left(\hypersum{\rho}{\sigma}b_{n}\right).
\end{align*}
In other words, if the Cauchy product of these hyperseries exists,
in general we have
\[
\left(\hypersum{\rho}{\sigma}a_{n}\right)*\left(\hypersum{\rho}{\sigma}b_{n}\right)\ne\left(\hypersum{\rho}{\sigma}a_{n}\right)\cdot\left(\hypersum{\rho}{\sigma}b_{n}\right).
\]

The classical Mertens' theorem also holds for hyperseries:
\begin{thm}
\label{thm:Cauchy}Assume that the hyperseries $\hypersum{\rho}{\sigma}a_{n}$
converges to $A$ and $\hypersum{\rho}{\sigma}b_{n}$ converges to
$B$. Assume that the former hyperseries converges absolutely. Then
their Cauchy product converges to $AB$.
\end{thm}

\noindent We first need the following
\begin{lem}
\label{lem:Cauchy}If $(b_{n})_{n}$ converges to $B$, then
\begin{equation}
\exists K\in\rcrho\,\forall N\in\hypNs:\ \left|\sum_{k=0}^{N}b_{k}-B\right|\le K.\label{eq:Cauchy}
\end{equation}
\end{lem}

\begin{proof}
Set $B_{N}:=\sum_{k=0}^{N}b_{k}\in\rti$ for all $N\in\hypNs$. The
assumption of convergence yields
\begin{equation}
\exists N_{1}\in\hypNs\,\forall M\in\hypNs_{\ge N_{1}}:\ \left|B_{M}-B\right|<1.\label{eq:conv1}
\end{equation}
Set $K=1+\sum_{k=0}^{N_{1}}|b_{k}|\in\rti$ and let $N\in\hypNs$.
We use dichotomy law as in \cite[Lem.~7.(ii)]{MTAG}, i.e.~we consider
$L\subzero I$ and two cases $N\ge_{L}N_{1}$ or $N\le_{L}N_{1}$.
We claim that $\left|B_{N_{\eps}}-B_{\eps}\right|\le K_{\eps}$ for
all $\eps\in L$ sufficiently small, where $N_{\eps}:=\nint{(N)}_{\eps}$,
$N_{1\eps}:=\nint{(N_{1})}_{\eps}$, $(b_{n})_{n}=[b_{n\eps}]_{\text{s}}\in\rcrhos$,
$B_{N_{\eps}}:=\sum_{k=0}^{N_{\eps}}b_{k\eps}$, $B=[B_{\eps}]$ and
$K_{\eps}:=1+\sum_{k=0}^{N_{1\eps}}\left|b_{k\eps}\right|$ (hence
$K=[K_{\eps}]$). We also note that, in general, $P$, $Q\in\hypNs$,
$P\le Q$ implies $\nint{(P)}_{\eps}\le\nint{(Q)}_{\eps}$ for $\eps$
small. In the first case $N\ge_{L}N_{1}$, we get $N_{\eps}\ge N_{1\eps}$
for $\eps$ small. Set $M_{\eps}:=N_{\eps}$ if $\eps\in L$ and $M_{\eps}:=N_{1\eps}$
otherwise, so that $M\in\hypNs_{\ge N_{1}}$ and from \eqref{eq:conv1}
we obtain $\left|B_{M_{\eps}}-B_{\eps}\right|<1$ for $\eps$ small
and, in particular
\[
\forall^{0}\eps\in L:\ \left|B_{N_{\eps}}-B_{\eps}\right|<1\le K_{\eps}.
\]
We claim the same conclusion also in the second case $N\le_{L}N_{1}$,
i.e.
\begin{equation}
\forall^{0}\eps\in L:\ N_{\eps}\le N_{1\eps}.\label{eq:2nd}
\end{equation}
In fact, we have $B_{N}=B_{N_{1}\vee N}-\sum_{k=N+1}^{N_{1}\vee N}b_{k}$
and hence
\begin{align}
\left|B_{N}-B\right| & \le\left|B_{N_{1}\vee N}-B\right|+\sum_{k=N+1}^{N_{1}\vee N}|b_{k}|<1+\sum_{k=0}^{N_{1}\vee N}|b_{k}|\nonumber \\
\forall^{0}\eps:\ \left|B_{N_{\eps}}-B_{\eps}\right| & \le1+\sum_{k=0}^{N_{1\eps}\vee N_{\eps}}|b_{k\eps}|.\label{eq:mah}
\end{align}
Inequalities \eqref{eq:2nd} and \eqref{eq:mah} prove the claim also
in the second case. Therefore, the dichotomy law yields $|B_{N}-B|\le K$
and $K$ does not depend on $N$.
\end{proof}
~
\begin{proof}[Proof of Thm.~\ref{thm:Cauchy}]
Define $A_{n}:=\sum_{i=0}^{n}a_{i}$, $B_{n}:=\sum_{i=0}^{n}b_{i}$
and $C_{n}:=\sum_{i=0}^{n}c_{i}$ for $n\in\hypNs$, and with $c_{i}:=\sum_{k=0}^{i}a_{k}b_{i-k}$
for $i\in\N$. Then $C_{n}=\sum_{i=0}^{n}a_{n-i}B_{i}$ because the
same equality holds $\eps$-wise for any representatives, and hence
$C_{n}=\sum_{i=0}^{n}a_{n-i}(B_{i}-B)+A_{n}B$. The idea is to estimate
\begin{align}
\left|C_{n}-AB\right| & =\left|\sum_{i=0}^{n}a_{n-i}(B_{i}-B)+(A_{n}-A)B\right|\le\nonumber \\
 & \le\sum_{i=0}^{N-1}|a_{n-i}||B_{i}-B|+\sum_{i=N}^{n}|a_{n-i}||B_{i}-B|+|A_{n}-A||B|,\label{eq:CauchySums}
\end{align}
and to use our assumption \eqref{eq:Cauchy} for the first summand.
We start from the second summand in \eqref{eq:CauchySums}: Since
$\hypersum{\rho}{\sigma}a_{n}$ converges absolutely to some $\bar{A}$
and $B_{n}$ converges to $B$, for all $q\in\N$ we can find $N\in\hypNs_{>0}$
such that

\[
\left|B_{n}-B\right|\le\frac{\diff{\rho}^{q}}{\bar{A}+1}
\]
for all $n\geq N$. Therefore, for the same $n\ge N$
\[
\sum_{i=N}^{n}|a_{n-i}||B_{i}-B|\le\diff{\rho}^{q}.
\]
First summand in \eqref{eq:CauchySums}: Since $\hypersum{\rho}{\sigma}a_{n}$
converges, $\left(a_{n}\right)_{n\in\hypNs}$ must converge to zero.
Hence for some $M\in\hypNs$ and for all $n\geq M$ and because of
Lem.~\ref{lem:Cauchy}, we have
\begin{equation}
\left|a_{n}\right|\le\frac{\diff{\rho}^{q}}{N(K+1)}\le\frac{\diff{\rho}^{q}}{N\left(\left|B_{i}-B\right|+1\right)}\quad\forall i\le N-1\label{eq:applAss}
\end{equation}
In particular, if $n\ge M+N$ and $i\le N-1$, then $n-i\ge n-N+1>M$
and we can apply \eqref{eq:applAss} with $n-i$ instead of $n$ obtaining
\[
\sum_{i=0}^{N-1}|a_{n-i}||B_{i}-B|\le\frac{\diff{\rho}^{q}N}{N}.
\]
Finally, since $A_{n}$ converges to $A$, for some $L\in\hypNs$
and for all $n\geq L$, we have

\[
\left|A_{n}-A\right|\le\frac{\diff{\rho}^{q}}{|B|+1}.
\]
Then, for all $n\geq\max\left(L,M+N\right)$ using the aforementioned
estimates, \eqref{eq:CauchySums} yields $\left|C_{n}-AB\right|\le3\diff{\rho}^{q}$.
\end{proof}

\section{Conclusions}

One of the main goals of the present paper is to show that when dealing
with non-Archimedean Cauchy complete rings, it can be worth to consider
summations over infinite natural numbers instead of classical series
indexed by $n\in\N$. As we proved for the Robinson-Colombeau ring
$\rcrho$, this allows us to extend numerous classical results which
do not hold using classical series in a non-Archimedean framework.
We tried to motivate in a clear way why we have to consider two gauges
$\sigma$ and $\rho$ and, for each result, what relationships we
have to consider among them. To help the reader to summarize the condition
on the gauges, we can use the following table
\noindent \begin{center}
\begin{tabular}{|c|c|}
\hline 
\noalign{\vskip0.7mm}
$\rcrhos$ & $\sigma$, $\rho$ arbitrary\tabularnewline
\hline 
\noalign{\vskip0.7mm}
$\rcrhou$ & $\sigma\ge\rho^{*}$\tabularnewline
\hline 
\noalign{\vskip0.7mm}
divergent hypersums & $\rcrhos$ and $\sigma\ge\rho^{*}$\tabularnewline
\hline 
\noalign{\vskip0.7mm}
geometric hyperseries & $\sigma\le\rho^{*}$\tabularnewline
\hline 
\noalign{\vskip0.7mm}
Taylor series exponential & $\sigma\le\rho^{*}$\tabularnewline
\hline 
\noalign{\vskip0.7mm}
$\eps$-wise convergence & $\exists\sigma\le\rho^{*}$ or $\sigma=\rho$\tabularnewline
\hline 
\noalign{\vskip0.7mm}
$p$-series & $\sigma$, $\rho$ arbitrary\tabularnewline
\hline 
\noalign{\vskip0.7mm}
direct comparison test & $\sigma$, $\rho$ arbitrary\tabularnewline
\hline 
\noalign{\vskip0.7mm}
limit comparison test & $\sigma$, $\rho$ arbitrary\tabularnewline
\hline 
\noalign{\vskip0.7mm}
root test & $\rcrhou$ and $\rho^{*}\le\sigma\le\rho^{*}$\tabularnewline
\hline 
\noalign{\vskip0.7mm}
ratio test & $\rcrhou$ and $\rho^{*}\le\sigma\le\rho^{*}$\tabularnewline
\hline 
\noalign{\vskip0.7mm}
alternating series test & $\rcrhou$ and $\sigma\ge\rho^{*}$\tabularnewline
\hline 
\noalign{\vskip0.7mm}
integral test & $\sigma$, $\rho$ arbitrary\tabularnewline
\hline 
\end{tabular}
\par\end{center}

\noindent Therefore, a possible summary could be: for meaningful examples
of convergent series, for all convergence tests, one can assume $\rho^{*}\le\sigma\le\rho^{*}$
and use both $\rcrhou$ or $\rcrhos$. For $\eps$-wise convergence,
to have larger domains (see Example \ref{exa:div}.\ref{enu:div6}),
one can assume $\sigma\le\rho^{*}$ and hence use $\rcrhos$. For
divergent hypersums, one must assume $\sigma\ge\rho^{*}$ and work
in $\rcrhos$. Intuitively, to get
\begin{equation}
\left|\sum_{n=0}^{N}a_{n}-s\right|\le\diff\rho^{q}\label{eq:goal}
\end{equation}
one is forced to ``sum up'' a sufficiently large number $N\in\hypNs$
of summands; how much this number has to be ``large'' depends on
the relationships between the two gauges $\sigma$ and $\rho$: if
$\sigma\le\rho^{*}$ then $\R_{\sigma}\supseteq\R_{\rho}$ and in
$\hypNs$ we surely have sufficiently ``large'' $\sigma$-hypernatural
numbers; whereas, if this is not the case, $\sigma$ can be ``too
large'' with respect to $\rho$ (i.e.~$\sigma^{-1}$ ``too small''
with respect to $\rho^{-1}$, e.g.~we could have $\sigma_{\eps}=\log\left(-\log\rho_{\eps}\right)^{-1}$,
see Example \ref{exa:GeomExpSeries},\ref{enu:geomSigma}) and we
cannot assure to obtain \eqref{eq:goal} because we do not sum up
a sufficient number of summands with respect to $\diff\rho^{q}$.

Since non-Archimedean rings are not Dedekind-complete, we have been
forced to substitute the least-upper-bound property using the weaker
Cauchy completeness. We think that several of the proofs we presented
here can be generalized to other non-Archimedean Cauchy complete settings.

Clearly, the present work opens the actual possibility to start the
study of hyper-power series, (real or complex) hyper-analytic generalized
smooth functions and sigma additivity of multidimensional integration
of generalized smooth functions (see \cite{TI}).

We also finally mention that the embedding of Schwartz's distributions
can be realized by regularization with an \emph{entire} Colombeau
mollifier (see e.g.~\cite{GKOS}). We can hence conjecture that this
embedding yields a hyper-analytic generalized smooth functions (but
note that the hyper-power series of classical non-analytic smooth
functions with a flat point would converge to zero under any invertible
infinitesimal). This conjecture would open the possibility to extend
the Cauchy-Kowalevski theorem to all Schwartz's distributions.

\medskip{}

\noindent \textbf{Acknowledgments} The authors would like to thank
the referee for several suggestions that have led to considerable
improvements of the paper.

\end{document}